\documentclass[smallextended,final]{svjour3}

\usepackage[english]{babel}
\usepackage[T1]{fontenc}

\usepackage{amssymb,amsmath,amscd}
\usepackage{graphicx}
\usepackage[colorlinks=true, allcolors=blue]{hyperref}
\usepackage{mathtools}
\usepackage{algorithm}
\usepackage[noend]{algpseudocode}
\usepackage{enumerate}
\usepackage[all]{xy}
\usepackage[version=3]{mhchem}
\usepackage{blkarray}

\newcommand{\R}{\mathbb{R}}
\newcommand{\Zint}{\mathbb{Z}}

\newcommand{\X}{\mathbb{X}}
\newcommand{\coker}[0]{\mathrm{coker}\,}
\newcommand{\im}[0]{\mathrm{im}\,}
\newcommand{\rank}[0]{\mathrm{rank}\,}

\newcommand{\low}{L}
\newcommand{\Cech}{\mathrm{\check{C}ech}}
\spdefaulttheorem{cond}{Condition}{\bf}{\rm}
\spdefaulttheorem{fact}{Fact}{\bf}{\rm}
\spdefaulttheorem{notation}{Notation}{\bf}{\rm}
\spdefaulttheorem{theoremext}{Theorem}{\bf}{\it}
\spdefaulttheorem{conj}{Conjecture}{\bf}{\it}

\begin{document}

\title{Field choice problem in persistent homology}
\author{Ippei Obayashi \and Michio Yoshiwaki}
\institute{
  Ippei Obayashi
  \at Center for Advanced Intelligence Project, RIKEN,
  Nihonbashi 1-chome Mitsui Building, 15th Floor,
  1-4-1 Nihonbashi, Chuo-ku, Tokyo 103-0027, Japan \\ \email{ippei.obayashi@riken.jp},
  ORCID: \url{https://orcid.org/0000-0002-7207-7280}
  \and
  Michio Yoshiwaki
  \at Center for Advanced Intelligence Project, RIKEN,
  Nihonbashi 1-chome Mitsui Building, 15th Floor,
  1-4-1 Nihonbashi, Chuo-ku, Tokyo 103-0027, Japan
  \at
  Osaka City University Advanced Mathematical Institute,
  3-3-138 Sugimoto, Sumiyoshi-ku, Osaka 558-8585, Japan
}

\date{Received: date / Accepted: date}

\maketitle

\begin{acknowledgement}
  This work was partially supported by 
  JSPS (Japan Society for the Promotion of Science)
  KAKENHI Grant Numbers JP 16K17638 and JP 19H00834, 
  JST (Japan Science and Technology Agency)
  CREST Grant Number JPMJCR15D3,
  JST PRESTO Grant Number JPMJPR1923,
  JST-Mirai Program Grant Number JPMJMI18G3,
  and Osaka City University Advanced Mathematical Institute (MEXT Joint Usage/Research Center on Mathematics and Theoretical Physics JPMXP0619217849).
  The authors are grateful to participants at the
  2nd JST math workshop on open problems
  for helpful discussions.
\end{acknowledgement}

\begin{abstract}
This paper tackles the problem of coefficient field choice in persistent homology.
When we compute a persistence diagram, we need to select a coefficient field before computation.
We should understand the dependency of the diagram on the coefficient field to facilitate 
computation and interpretation of the diagram.
We clarify that the dependency is strongly related to the torsion of $\Zint$ relative homology
in the filtration. We show the sufficient and necessary conditions of
the independence of coefficient field choice.  An efficient
algorithm is proposed to verify the independence. In a numerical experiment with the algorithm,
a persistence diagram rarely changes even when the coefficient field changes if we consider a
filtration in $\R^3$.
The experiment suggests that, in practical terms, changes in the field coefficient will not change persistence diagrams when the data are in $\R^3$.

\end{abstract}

\section{Introduction}
Topological data analysis (TDA) \cite{eh,carlsson} is, as the name suggests, the application of
topology to data analysis. Persistent homology \cite{elz,zc} is one of the most
important tools for TDA.
In persistent homology, by encoding information on length scales in filtrations,
we can capture characteristic geometric features with multiple length scales.
By using filtrations, persistent homology is also robust to noise \cite{cohen2007stability,chazal2009proximity}.
Homology itself is translation and rotation invariant, and so
persistent homology is similarly invariant. These properties are
suitable for the analysis of shapes of data, and
persistent homology is applied in various practical data analysis contexts in domains such as
biology~\cite{virus}, image processing~\cite{Hu2019},
and materials science~\cite{Hiraoka28062016,granular,PhysRevE.95.012504,Kimura2018}.

To describe our problem, we first define persistent homology. Persistent homology
is defined on a filtration, an increasing sequence of topological spaces.
We consider the following filtration:
\begin{align*}
  \X & = (X_t)_{t \in T}, \\
  X_t & \subset X_s \mbox{ if $t \leq s$}, \\
  X_0 &= \emptyset,
\end{align*}
where $T$ is $\{0, 1, \cdots, N\}$ or $\R_+$. 
The $q$th \emph{persistent homology} $H_q(\X; \Bbbk)$
with a coefficient ring $\Bbbk$ is defined as follows:
\begin{align*}
  H_q(\X; \Bbbk) &\mbox{ is a pair of } \\
                 & \{H_q(X_t; \Bbbk)\}_{t \in T}, \\
                 & \{\phi_s^t:H_q(X_s;\Bbbk) \to H_q(X_t; \Bbbk) \}_{s \leq t},
\end{align*}
where $\phi_s^t:H_q(X_s;\Bbbk) \to H_q(X_t; \Bbbk)$ is the homology map
induced by the inclusion map $X_s \hookrightarrow X_t$.
In standard homology theory, we use $\Zint$ as a coefficient ring
since the universal coefficient theorem ensures
that $\Zint$-homology provides the most information about homology.
However, in the theory of persistent homology a field is used instead of $\Zint$ since
the interval decomposition described below is crucial for analysis of persistent homology
and the decomposition is guaranteed \emph{only when $\Bbbk$ is a field}.
Indeed, the structural theory of persistent homology ensures the existence and 
uniqueness of the following decomposition of $H_q(\X; \Bbbk)$ called 
the \emph{interval decomposition}
if $\Bbbk$ is a field:
\begin{theorem}
For a filtration $\X  = (X_t)_{t \in T}$ of topology spaces $X_t$ where $T=\{0, 1, \cdots, N\}$, the $q$-th persistent homology $H_q(\X; \Bbbk)$ has the following unique decomposition.  
\begin{align*}
  H_q(\X; \Bbbk) &= \bigoplus_{i=1}^L I(b_m, d_m), \\
  I(b, d) &= \{\varphi_s^t: U_s \to U_t\}_{s \leq t}, \\
  U_s &= \left\{
  \begin{array}{ll}
    \Bbbk & \mbox{ if $b \leq s \leq d$}, \\
    0 & \mbox{ otherwise},
  \end{array}\right. \\
  \varphi_s^t &= \left\{
  \begin{array}{ll}
    \mathrm{id} & \mbox{ if $b \leq s \leq t \leq d$}, \\
    0 & \mbox{ otherwise},
  \end{array}\right. 
\end{align*}
where $0 < b_m \leq d_m \leq \infty$.
This $I(b, d)$ is called an \emph{interval indecomposable}.
\end{theorem}
This theorem depends on the fact that $\Bbbk[z]$, the polynomial ring
with a field coefficient, is a PID~\cite{zc}, and so this theorem does not hold
for $\Bbbk = \Zint$.

When this interval decomposition is given, we define the $q$th
\emph{persistence diagram} (PD) $D_q(\X; \Bbbk)$ as a multiset of pairs of
endpoints of the intervals. That is, $D_q(\X; \Bbbk) = \{(b_m, d_m)\}_{i=1}^L$.
Each pair is called a \emph{birth-death pair}. $b_m$ and $d_m$ are called \emph{birth time} and \emph{death time} of the birth-death pair $(b_m, d_m)$, and
$d_m - b_m$ is called a \emph{lifetime} of the pair. Since a birth-death pair
with a long lifetime corresponds to a ``stable'' homological structure
in the filtration, we can use lifetimes to compare the significance of birth-death pairs.

Normally, we choose $\Bbbk$ as one of $\R, \mathbb{Q}$, and $\Zint_p = \Zint/p\Zint$ for a prime $p$.
$\Zint_2$ is most often used since it is amenable to a fast algorithm
and an intuitive interpretation.
Here we face the problem of the choice of $\Bbbk$.
If any $\Bbbk$ gives the same PD, 
there is no problem. However, this is not practical, because the dimensions of homology vector spaces
for the same topological space are different when the $\Zint$-homology group of the space has non-zero
torsion. If a Klein bottle appears in a filtration, the PDs 
for $\Bbbk = \Zint_2$ and $\Bbbk = \R$ are clearly different.
For analysis of persistent homology with
torsions, Boissonnat and Maria
\cite{BJDMC} proposed
an efficient algorithm
to compute PD for multiple coefficient fields
by utilizing the Chinese remainder theorem.
Then, the following questions naturally arise.
\begin{itemize}
    \item What condition ensures the independence of the choice of the field $\Bbbk$?
    \item If there is such a condition, is there an efficient algorithm to check it?
    \item How often does $D_q(\X; \Bbbk)$ change as the field $\Bbbk$ changes?
    \item When $D_q(\X; \Bbbk)$ changes depending on the choice of $\Bbbk$, how does $D_q(\X; \Bbbk)$ change?
\end{itemize}
In this paper, we offer complete
answers for the first and second questions, and partial answers
for the third and fourth questions. 

\subsection{Results}
To describe the results of the paper, we give some assumptions.
We always assume the finiteness of the filtration.
A filtration is \emph{finite}
if $X = \cup_t X_t$ is a finite simplicial/cell/cubical complex.
This condition ensures the existence and uniqueness of the interval
decomposition~\cite{zc}. This assumption is reasonable since
an infinite filtration cannot be represented on a computer and
we cannot use such a filtration in practical applications.
To consider field choice problems, we always restrict the
candidates of a field to $\mathbb{C}, \R, \mathbb{Q},$ and
$\Zint_p$ for a prime $p$. 

\begin{question}\label{q:cond}
  When is $D_q(\X; \Bbbk)$ independent of the choice of the field $\Bbbk$?
\end{question}

To consider Question~\ref{q:cond}, 
it is desirable for the following proposition to hold since
$H_q(X_t; \Zint)$ is often free for every $t$ in practical cases.
\begin{conj}[Incorrect!] 
  If $H_q(X_t; \Zint)$ is free for every $t \in T$, then
  the persistent homology $H_q(\X; \Bbbk)$ has the same decomposition for any field
  $\Bbbk$.
\end{conj}
However, we have a counterexample of this proposition (Fig.~\ref{fig:mobius}).
\begin{example} \label{exm:mobius}
Let $M$ be a M\"obius strip
and $\partial M$ be its boundary. Both $H_1(\partial M; \Zint)$ and $H_1(M; \Zint)$ are
isomorphic to $\Zint$, and the homomorphism
\begin{align*}
  H_1(\partial M; \Zint) \to H_1(M; \Zint)
\end{align*}
is isomorphic to
\begin{align*}
  n \in \Zint \mapsto 2n \in \Zint
\end{align*}
and the interval decomposition on $\R$ and  $\Zint_2$ gives
the different decomposition as follows:
\begin{align*}
  &\X : X_0 = \emptyset \subset X_1 = \partial M \subset X_2 = M, \\
  &H_1(\X; \Zint_2) = I(1, 2) \oplus I(2, \infty), \\
  &H_1(\X; \R) = I(1, \infty). 
\end{align*}
In this example, both $H_1(\partial M; \Zint)$ and $H_1(M; \Zint)$
are free, but $H_1(M, \partial M; \Zint) \simeq \Zint_2$ and this is not free.
This fact is key to the different diagrams.
Section~\ref{sec:example} shows some other examples.

\begin{figure}[hbtp]
  \centering
  \includegraphics[width=0.7\hsize]{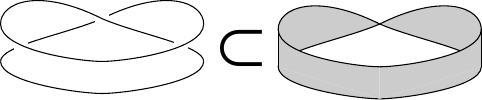}
  \caption{M\"obius strip and its boundary}
  \label{fig:mobius}
\end{figure}
\end{example}

We present the following theorem.
\begin{theorem}\label{thm:allfree}
  $D_q(\X; \Bbbk)$ is independent of the choice of $\Bbbk$ if
  $H_q(X_n, X_m; \Zint)$ is free for any $0\leq m < n \in T$ and
  $H_{q-1}(X_n; \Zint)$ is free for any $n \in T$. 
\end{theorem}
This theorem yields the following corollaries.
\begin{corollary}\label{cor:M_0}
  $D_0(\X; \Bbbk)$ is always independent of the choice of $\Bbbk$.
\end{corollary}
\begin{corollary}\label{cor:M_munus_1}
  When $\X$ is a filtration of finite cell/simplicial/cubical complexes embedded
  in $\R^M$, the $(M-1)$th persistent homology gives the same
  PD among any fields $\Bbbk$.
\end{corollary}
Corollary~\ref{cor:M_0} derives from the fact that $H_{-1}(\cdot) = 0$ and $H_0(X_n, X_m; \Zint)$
is free for any $n > m$.
Corollary~\ref{cor:M_munus_1} is proved in Section~\ref{subsec:M_minus_1}.
The above two corollaries ensure that if a filtration is embedded in $\R^2$, 
all non-trivial persistence diagrams $D_0$ and $D_1$ do not depend on the choice of the 
coefficient field. This is the case of alpha filtrations for a 2D point cloud.

We also have the following theorem which provides a sufficient condition
for the freeness of $H_q(X_n, X_m; \Zint)$.
\begin{theorem}\label{thm:allfree2}
  For a given $q$, 
  $H_q(X_n, X_m; \Zint)$ are free for any $0\leq m < n \in T$ if
  $D_q(\X; \Bbbk)$ is independent of the choice of $\Bbbk$ and
  $H_{q-1}(X_n; \Zint)$ is free for any $n \in T$.
\end{theorem}

From the above two theorems, we have the following corollary.
\begin{corollary}\label{cor:allfree3}
  Let $M$ be a non-negative integer.
  $D_q(\X; \Bbbk)$ is independent of the choice of $\Bbbk$ for
  all $q = 0, \ldots, M$ if and only if
  $H_q(X_n, X_m; \Zint)$ are free for any $0\leq m < n \in T$, and $q = 0, \ldots, M$.
\end{corollary}

From the above discussion the following question arises.
\begin{question}
  Is there an efficient algorithm for checking the condition of
  Corollary~\ref{cor:allfree3}?
\end{question}
Such an algorithm would be useful to provide information as to whether we should be concerned about field choice.
Of course, we can compute relative homology groups for all $m < n$ on a computer,
but that would be cumbersome and inefficient
because the number of possible pairs $(m, n)$ is $(N+1)N/2$.
Thus, the computation cost (time complexity)
is $O(N^2 G)$, where
$G$ is the average cost of computing $H_q(X_n, X_m; \Zint)$.
It is known that the time complexity of computing a PD is $O(G)$
\footnote{To compute the torsion subgroup of a homology group, we need to compute the Smith normal form of the boundary matrix, and the computational cost of the Smith normal form is $O(n^{\theta})$ in the worst case where $n$ is the number of simplices and $\theta \approx 2.376$ is a constant. The constant $\theta$ derives from the time complexity of the multiplication of two $n \times n$ matrices. The time complexity of persistent homology is also $O(n^{\theta})$ in the worst case. See \cite{Milosavljevic:2011:ZPH:1998196.1998229,Storjohann} for further details.}.

To describe the algorithm, we assume the following condition.
\begin{cond}\label{cond:finite}
  $X = \{\sigma_1, \sigma_2, \ldots, \sigma_N\}$ is a finite simplicial, cubical, or cell
  complex and the subset $X_k = \{\sigma_1, \ldots, \sigma_k\}$ is a sub-complex of
  $X$.
\end{cond}
With this setting, we consider the filtration of complexes
$\X: \emptyset = X_0 \subset X_1 \subset \cdots \subset X_N$.
Since the filtration is finite, we can transform the persistence decomposition problem
into the problem under Condition~\ref{cond:finite}. 

The following theorem is proved in Section~\ref{sec:algorithm}. 
\begin{theorem}\label{thm:allfreealg}
  There is an algorithm for judging the condition in Corollary~\ref{cor:allfree3}
  whose
  time complexity is the same as the algorithm for computing a PD.
\end{theorem}
The algorithm is shown in Algorithm~\ref{alg:allfree} in Section~\ref{sec:algorithm}.
In Section~\ref{sec:implementation}, we apply the algorithm to some examples shown in Section~\ref{sec:example} and demonstrate that
it performs well. A performance benchmark is also covered in that section.

It is noteworthy that Dionysus2\footnote{\url{https://mrzv.org/software/dionysus2/index.html}}, the persistent homology software developed by Morozov, has a similar algorithm\footnote{Omni-field Persistence: \url{https://mrzv.org/software/dionysus2/tutorial/omni-field.html}}. However, we cannot find any description of the mathematical background of the algorithm. It is likely that the program keeps tracks of the torsion information in a rational number and the mathematical mechanism of the program is similar to Algorithm~\ref{alg:allfree}.

We now pose the following additional question.
\begin{question}
  How often do we face filtrations with non-trivial torsion subgroups?
\end{question}
We can construct such an example by a M\"obius strip as shown in Example~\ref{exm:mobius}, but
would we often face such a filtration?
To demonstrate the probability of torsions
we conduct a numerical experiment for random data in $\R^3$.
From this experiment, we show that filtrations with non-trivial
torsion subgroups are very rare This suggests that, in practical terms,
if the data are in $\R^3$,
we do not need to be particularly concerned about the torsion problem.
We also conduct another numerical experiment for random filtrations in high dimensional
simplex. This second experiment shows that the filtrations with non-trivial torsion subgroups
are to be expected when the space is high dimensional. 

The following question is also important.
\begin{question}
  When $D_q(\X; \Bbbk)$ changes depending of the choice of $\Bbbk$,
  how does $D_q(\X; \Bbbk)$ change?
\end{question}

In the above example about a M\"obius strip, a long interval $I(1, \infty)$ is
split into two shorter intervals, $ I(1, 2)$ and $I(2, \infty)$, when
$\Bbbk$ changes from $\R$ to $\Zint_2$.
From the example, we expect that a long interval indecomposable tends to be split
into shorter intervals when $\Bbbk$ changes from $\R$ to $\Zint_p$.
The following theorem proved in Section~\ref{sec:convex} partially answers the question.
\begin{theorem}\label{thm:r-wasserstein}
  Let $q$ be a positive integer.
  Assume that $H_q(X_t; \Zint)$ and $H_{q-1}(X_t; \Zint)$
  are free for all $t$ and $H_q(\cup_t X_t) = 0$.
  Let $f$ be a $C^2$ convex function on $[0, \infty)$ with $f(0) = 0$.
  Then the following inequality holds:
  \begin{align*}
    \sum_{(b, d) \in D_q(\X; \R)} f(d - b) \geq
    \sum_{(b, d) \in D_q(\X; \Zint_p)} f(d - b).
  \end{align*}
  When $f$ is strictly convex, 
  the equality holds if and only if $D_q(\X; \R) = D_q(\X; \Zint_p)$ for all $p$.
\end{theorem}
For $f(x) = x^r$ with $r > 1$, the inequality means
\begin{align*}
    W_r(D_q(\X; \R), \emptyset) \geq W_r(D_q(\X; \Zint_p), \emptyset),
\end{align*}
where $W_r$ is the $r$-Wasserstein distance.
In some sense, the $r$-Wasserstein distance from the empty diagram indicates
the information richness of the diagram. Therefore, 
$D_q(\X; \R)$ contains richer information than $D_q(\X; \Zint_p)$ under the condition
of the theorem.

\subsection{Related works}

 \cite{BJDMC} proposed an efficient algorithm to simultaneously compute PDs for multiple coefficient fields by utilizing the Chinese remainder theorem.

\cite{Carlsson2008} analyzed the space of small (3x3) image patches collected from natural images. In that study, a Klein bottle, a topological structure with non-trivial torsion subgroup in its homology subgroup, played an important role. They developed a theoretical model for the high-density 2-dimensional submanifold showing that it has the topology of the Klein bottle in the space of image patches. Each image patch is represented as a point in $S^7$ and the points are analyzed by using persistent homology. Those authors used persistent homology with $\Zint_2$ coefficient field to analyze the data so it was difficult to directly find evidence of the Klein bottle. Therefore, they developed some techniques to verify the Klein bottle.

\cite{doi:10.1080/10586458.2018.1473821} experimentally investigate the torsion subgroup in random $d$-complex which is a subcomplex of $(n-1)$-simplex.  They showed that a randomly generated simplicial complex often has a non-trivial torsion subgroup for relatively large $n$ (for example, $n=75$).

\cite{Newman2019} estimated the minimum number of vertices of simplicial complex whose homology group has a desired torsion part. He gives upper and lower bounds of the number in the following form:
\begin{equation}
    c_d(\log |G|)^{1/d} \leq (\text{the minimum number of vertices}) \leq C_d(\log|G|)^{1/d}
\end{equation}
where $G$ is the desired torsion group, $d$ is the degree of homology, and $c_d, C_d$ are two constants which depend only on $d$. The result did not say anything about the frequency of the appearance of a non-trivial torsion subgroup, but it suggested how difficult it is to find such a subgroup of a homology group. 

\cite{optimal-Day} showed that the relationship between the existence of a non-trivial torsion subgroup and the computational complexity of a kind of optimization problem on homology algebra. Their result showed that the problem essentially becomes harder if some relative homology groups in the given complex have non-trivial torsion subgroups. The relation between that study and our results is discussed in the conclusion section.

Where the persistence diagram is dependent on the field of coefficients, our main motivating example is the boundary of a M\"obius band (see Example~\ref{exm:mobius}).
This has already been considered in the topological time series analysis literature \cite[Section 3]{perea2015sliding}. 
Indeed, Perea and Harer constructed a {\em sliding-window point cloud} $SW(g)$ for some function $g$ on $\mathbb{R}$ such that the normalization of $C(SW(g))$, where $C$ is the centering map,  can be isometrically identified with $\widetilde{\phi}$, whose image can be realized as the boundary of a M\"obius band. 
Those authors then verified that the persistence diagram of $\widetilde{\phi}$ is dependent on the field of coefficients.

The universal coefficient theorem for ordinary homology plays an important role in this paper. We note that the theorem for persistence modules has already been developed by \cite{bubenik2019homological}. 
More generally, the Kunneth theorem for persistence modules has been developed in the literature \cite{bubenik2019homological,gakhar2019k,polterovich2017persistence}.

\subsection{Organization of the paper}

The remainder of the paper is organized as follows. 
Section~\ref{sec:ph} reviews the basic concepts of persistent homology.
Section~\ref{sec:example} shows some examples which exhibit the dependency of 
PDs to their coefficient fields.
Section~\ref{sec:proof-allfree} and Section~\ref{sec:proof-allfree2} prove 
Theorem~\ref{thm:allfree} and Theorem~\ref{thm:allfree2}.
Section~\ref{sec:algorithm} presents an efficient algorithm which permits judgement and the proof which testifies to the correctness of the algorithm.
Section~\ref{sec:implementation} introduces an implementation of the algorithm in HomCloud.
This section also shows the performance benchmark. Section~\ref{sec:probability} presents
numerical experiments to measure the probability of the appearance of non-trivial 
torsions in random filtrations.
Section~\ref{sec:convex} contains the proof of Theorem~\ref{thm:r-wasserstein} and, finally, conclusions are offered in Section~\ref{sec:conclusion}.

\section{Persistent homology}\label{sec:ph}
In this section, we prepare some fundamental concepts for persistent homology.

\subsection{Filtrations}

A filtration is an increasing sequence of topological spaces. One typical filtration is 
the union of $r$-balls constructed from a pointcloud in $\R^{M}$. For a pointcloud, a set of finite
points $\{x_i\}$, $X_r$ is defined as
\begin{equation}
    X_r = \cup_i B_{x_i}(r),
\end{equation}
where $B_{x}(r)$ is the closed ball whose center is $x$ and radius is $r$\footnote{Open balls are usually used to define \v{C}ech filtrations, but in this paper we use closed balls to simplify the definition of persistence Betti numbers. The nerve theorem holds for both open and closed balls.}.
The sequence of $X_r$ parameterized by $r$,
$\{X_r\}_{r \geq 0}$, is obviously a filtration. This filtration is used to investigate
the shape formed by the pointcloud. 

For a practical application of persistent homology, we usually use finite simplicial or
cubical filtrations since such filtrations are practical to consider on a computer.
One well-known filtration is the \v{C}ech filtration. The \emph{\v{C}ech complex} $\Cech(P, r)$ of a pointcloud
$P = \{x_i\}$ with radius parameter $r \geq 0$ is defined as follows:
\begin{equation}
    \Cech (P, r) = \{ \{x_{i_1}, \ldots, x_{i_k}\} \subset P \mid 
    \bigcap_{n=1}^k B_{x_{i_n}}(r) \not = \emptyset \}.
\end{equation}
The filtration $\{\Cech(P, r)\}_{r \geq 0}$ is called a \emph{\v{C}ech filtration}. From the nerve theorem,
$\Cech(P, r)$ is homotopy equivalent to $\cup_i B_{x_i}(r)$ and we can use the \v{C}ech filtration to investigate
the union of $r$-balls. There are many simplices in a \v{C}ech complex for a large
pointcloud and we usually use an \emph{alpha complex} \cite{alphashape1,alphashape2}
instead since the alpha complex is
homotopy equivalent to the \v{C}ech complex and the number of simplices of the alpha complex
is much smaller than the \v{C}ech complex. The alpha complex has another advantage in that it can be embedded in $\R^M$ but such embedding is impossible for
the \v{C}ech complex.

When a filtration is finite, it is essentially time-discrete even if $T = \R_+$.
Therefore we assume $T = \{0, \ldots, N\}$ for the proofs of this paper
except Theorem~\ref{thm:r-wasserstein}. In addition, under this assumption,
it is straightforward to configure a filtration satisfying Condition~\ref{cond:finite} by ordering simplices
appropriately; hence, we can assume the condition without loss of generality.
Since Condition~\ref{cond:finite} is useful to describe algorithms, we sometimes 
assume this and consider the filtration
$X_0 \subset X_1 \subset \cdots \subset X_N$ where
$X_k = \{\sigma_1, \ldots, \sigma_k\}$.

\subsection{Computation of a persistence diagram}
Under Condition~\ref{cond:finite}, Algorithm~\ref{alg:pd} computes the PD of the filtration~\cite{elz,zc,Otter2017}. To simplify the algorithm,
all simplices of all dimensions are mixed and
in the output all birth-death pairs of all degrees are also mixed. In this algorithm,
$L_B(j)$ means
\begin{equation}
  L_B(j) = \left\{
    \begin{array}{ll}
      \max \{i \mid B_{ij} \not = 0\} & \text{if column $j$ of $B$ is nonzero},\\
      -\infty & \text{if column $j$ of $B$ is zero},
    \end{array}
  \right. \label{eq:LB}
\end{equation}
where $B$ is a matrix and $j$ is an integer.

Furthermore, in this algorithm, matrix $B$ is reduced from left column to right column.
After terminating the algorithm, the PD is computed as follows:
\begin{equation}\label{eq:pd}
  \begin{aligned}
    D(\X) =& \{ (L_{\hat{B}}(j), j) \mid L_{\hat{B}}(j) \not = -\infty \} \\
    &\cup
    \{(j, \infty) \mid L_{\hat{B}}(j) = -\infty \text{ and } \forall i, L_{\hat{B}}(i) \not = j\}, 
  \end{aligned}
\end{equation}
where $\hat{B}$ is the matrix returned by the algorithm.
The $q$th PD is given from $D(\X)$ as follows:
\begin{equation}
    D_q(\X) = \{(i, j) \in D(\X) \mid \dim \sigma_i = q\}.
\end{equation}

\begin{algorithm}[ht]
  \caption{Algorithm to compute persistence diagrams}\label{alg:pd}
  \begin{algorithmic}
    \State $B \leftarrow $ the boundary matrix with respect to the basis $\{\sigma_1, \ldots, \sigma_N\}$ 
    \For{$j=1, \ldots, N$}
      \While{there exists $i < j$ with $L_B(i) = L_B(j) \not = -\infty$ }
         \State let $s = B_{L_B(j), j} / B_{L_B(i), i}$  
         \State add $(-s) \times (\textrm{column}\  i)$ to column $j$ of $B$  \Comment left-to-right reduction
      \EndWhile
    \EndFor
    \State \Return B
  \end{algorithmic}
\end{algorithm}

A justification for the algorithm is provided in Appendix \ref{sec:proof-alg1}.
Indeed, Theorem~\ref{thm:allfreealg} shows that the algorithm for judging the condition of Corollary~\ref{cor:allfree3}
is given by restricting Algorithm~\ref{alg:pd} to integer coefficients. Therefore, the time complexity of the Theorem~\ref{thm:allfreealg} algorithm is as per  Algorithm~\ref{alg:pd}.

\subsection{Persistent Betti number}

From the definition of a PD, we have the following relationship
between the map $H_q(X_m; \Bbbk) \to H_q(X_n; \Bbbk)$ and a PD:
\begin{equation}\label{eq:persistence-betti-number}
  \begin{aligned}
    \beta_m^n(\Bbbk) :=& \rank\left(H_q(X_m; \Bbbk) \to H_q(X_n; \Bbbk)\right) \\
    =& \#\{ (b, d) \in D_q(\X; \Bbbk) \mid b \leq m \leq n < d \}.
  \end{aligned}
\end{equation}
This $\beta_m^n(\Bbbk)$ is called a \emph{persistent Betti number} or a \emph{rank invariant}. 
Hence, the following identity holds:
\begin{equation}\label{eq:multiplicity-1}
  (\text{multiplicity of $D_q(\X; \Bbbk)$ at $(b, d)$})
  = \beta^{d-1}_b(\Bbbk) - \beta^{d-1}_{b-1}(\Bbbk) - \beta^{d}_{b}(\Bbbk)
  + \beta^{d}_{b-1}(\Bbbk).
\end{equation}

When $d = \infty$, the following equation holds instead:
\begin{equation}\label{eq:multiplicity-2}
  (\text{multiplicity of $D_q(\X; \Bbbk)$ at $(b, \infty)$})
  = \beta^{N}_b(\Bbbk) - \beta^{N}_{b-1}(\Bbbk).
\end{equation}

The next lemma follows directly from the foregoing.
\begin{lemma}\label{lem:pbn}
  $D_q(\X, \Bbbk) = D_q(\X, \Bbbk')$ if and only if
  $\beta_m^n(\Bbbk) = \beta_m^n(\Bbbk')$ for all $0 \leq m \leq n \leq N$.
\end{lemma}

\subsection{Universal Coefficient Theorem}

The universal coefficient theorem is fundamental for homology theory
and plays an important role in this paper. We review the theorem here to foreground what follows.

The \emph{universal coefficient theorem for homology} is as follows~\cite{AT}.
\begin{theoremext}
  Let $X$ be a topological space, $\Bbbk$ a ring, 
  and $q \geq 0$. The following sequence is a natural short
  exact sequence:
  \begin{equation}
      \begin{aligned}
        0 \to H_q(X; \Zint) \otimes \Bbbk \to H_q(X; \Bbbk)  \to 
        \mathrm{Tor}(H_{q-1}(X; \Zint), \Bbbk) \to 0.
      \end{aligned}
  \end{equation}
  Furthermore, this sequence splits, though not naturally.
\end{theoremext}

We use the above theorem in the following form.
\begin{theoremext}\label{thm:ucm}
Let $X$ and $Y$ be topological spaces, $f:X \to Y$ a continuous map, 
$\Bbbk$ a ring, and $q \geq 0$.
If $H_{q-1}(X; \Zint)$ and $H_{q-1}(Y; \Zint)$ are free, the following commutative diagram holds:
\begin{equation}
 \vcenter{
 \xymatrix{
    H_q(X; \Zint)\otimes \Bbbk \ar[r]^{\simeq} \ar[d]_{f_*\otimes \mathrm{id}_{\Bbbk}} & H_q(X; \Bbbk) \ar[d]_{f_*} \\
    H_q(Y; \Zint)\otimes \Bbbk  \ar[r]^{\simeq} & H_q(Y; \Bbbk). \ar@{}[lu]|{\circlearrowright}
  }
  }
\end{equation}
\end{theoremext}
This theorem states that the induced map
$f_* : H_q(X; \Bbbk) \to H_q(Y; \Bbbk)$ is completely described by
$f_*\otimes \mathrm{id}_\Bbbk : H_q(X; \Zint) \otimes \Bbbk \to H_q(Y; \Zint)\otimes \Bbbk$ if 
$H_{q-1}(X; \Zint)$ and $H_{q-1}(Y; \Zint)$ are free.
We use the theorem for an inclusion map between simplicial/cell/cubical complexes. 

\section{Examples of diagrammatic changes induced by coefficient field changes}\label{sec:example}

In this section, we give some examples of persistent homology, whose interval decomposition depends on the choice of coefficient field.

\begin{example} \label{exm:bouquet}
Let $S^1$ be a circle.
We consider a filtration $\X : \emptyset \to S^1 \xrightarrow{f} S^1 \vee S^1 \xrightarrow{g} S^1$, where $S^1 \vee S^1$ is a bouquet of $2$-circles, $H_1(f)=\begin{pmatrix}1 \\ 1 \end{pmatrix}$ 
and $H_1(g)=\begin{pmatrix}1&1 \end{pmatrix}$.
By taking the 1st homology of this filtration, 
we obtain the 1st persistent homology
$$
H_1 (\X ; \Zint)=
\xymatrix{
0\ar[r] & \Zint \ar[r]^{\begin{pmatrix}1 \\ 1 \end{pmatrix}} & \Zint^2 \ar[r]^{\begin{pmatrix}1&1 \end{pmatrix}} & \Zint,
}
$$
with a coefficient ring $\Zint$.
Then $H_1 (\X ; \Zint_2) = H_1(\X ; \Zint) \otimes_\Zint \Zint_2$ has the interval decomposition $I(1,2) \oplus I(2,3)$.
On the other hand, $H_1 (\X ; \R) = H_1 (\X ; \Zint) \otimes_\Zint  \R$ has the interval decomposition  $I(1,3) \oplus I(2,2)$.
Thus, the interval decomposition of the 1st persistent homology of $\X$ depends on the choice of coefficient field.
\begin{figure}[h]
  \centering
  \includegraphics[width=1.0\hsize]{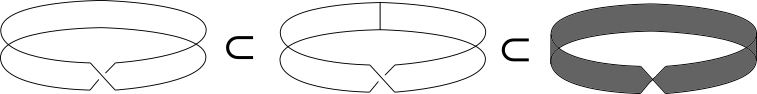} \\
  \caption{Visualization of $\X : \emptyset \to S^1 \xrightarrow{f} S^1 \vee S^1 \xrightarrow{g} S^1$}
\end{figure}
\end{example}

Note that if we consider a bouquet of $p$-circles for a prime $p$, 
then we obtain the 1st persistent homology, which has different decompositions over $\Zint_p$ and $\R$. 

By using Example \ref{exm:bouquet}, we can consider the 1st persistent homology, whose interval decomposition depends on the choice of characteristic $p>0$.

\begin{example} \label{exm:charp}
Let 
$$
M=
\xymatrix{
0 \ar[r] &
\Zint \ar[r]^{\begin{pmatrix}1 \\ 1 \end{pmatrix}} & \Zint^2 \ar[r]^{\begin{pmatrix}1&1 \end{pmatrix}} & \Zint \ar[r]^{\begin{pmatrix}1 \\ 1 \\1 \end{pmatrix}}  & \Zint^3 \ar[r]^{\begin{pmatrix}1 & 1 & 1 \end{pmatrix}} & \Zint
}
$$
be the 1st persistent homology.  
Then $M$ has the following interval decomposition:
$$
M \cong \left\{
\begin{array}{cc}
I(1,2) \oplus I(2,5) \oplus I(4,4)^2     & \text{if } p=2 \\
I(1,4) \oplus I(2,2) \oplus I(4,4) \oplus I(4,5)     & \text{if }  p=3 \\
\end{array}\right.
.
$$
\begin{figure}[h]
  \centering
  \includegraphics[width=1.0\hsize]{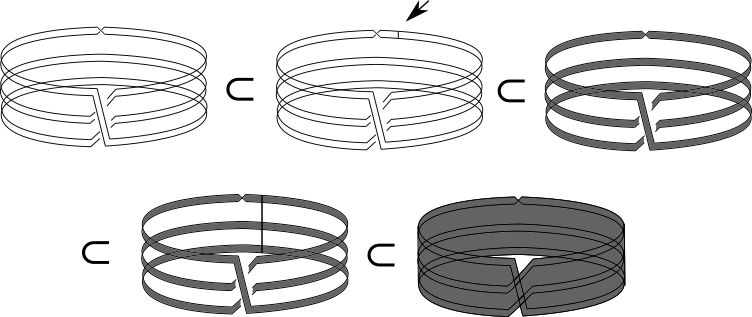} \\
  \caption{Visualization of $\X$ with $M=H_1 (\X)$}
\end{figure}
\end{example}

Other examples are double and triple loop pointclouds.
Figure~\ref{fig:doubletripleloop} (a) shows the double loop pointcloud. The pointcloud is located on the boundary
of a M\"obius strip. We compute the 1st PDs of the double loop pointcloud
with fields $\Zint_2, \Zint_3,$ and $\Zint_5$. The alpha filtration of the pointcloud is used for the computation.
The diagrams are shown in Figure~\ref{fig:doubletripleloop} (b), (c), and (d). Note that  (c) and (d) are the same diagram. The difference between (b) and (c) is only two birth-death pairs.
Figure~\ref{fig:doubletripleloop} (e) shows the triple loop pointcloud and (f), (g), and (h) show
the 1st PDs of the triple loop pointcloud with fields $\Zint_2, \Zint_3,$ and $\Zint_5$.
To be expected, (f) and (h) are the same diagram and (g) is different from (f) and (h).

\begin{figure}[h]
  \centering
  \includegraphics[width=1.0\hsize]{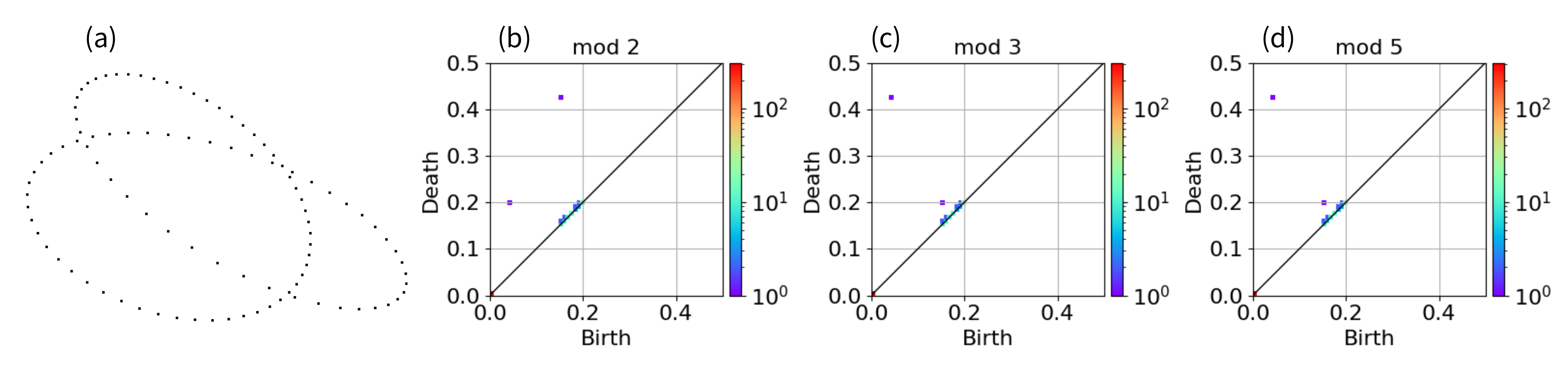} \\
  \includegraphics[width=1.0\hsize]{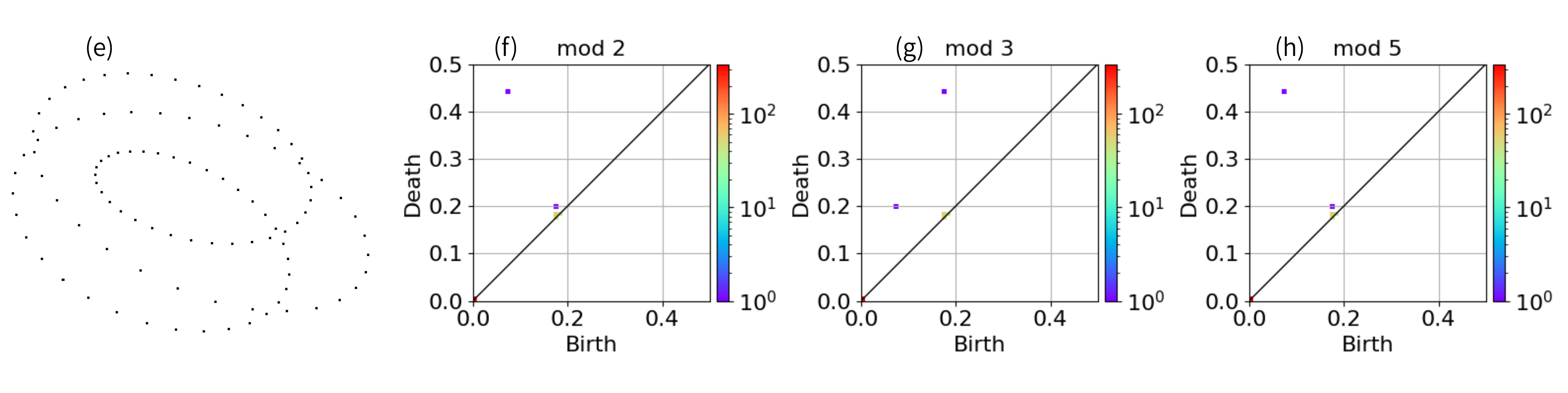} 
  \caption{1st PDs with various fields for double and triple loop pointclouds.
    (a) A double loop point cloud. (b) The PD of the double loop with $\Zint_2$.
    (c) The PD of the double loop with $\Zint_3$. (d) The PD of the double loop with $\Zint_5$.
    (e) A triple loop point cloud. (f) The PD of the triple loop with $\Zint_2$.
    (g) The PD of the triple loop with $\Zint_3$. (h) The PD of the triple loop with $\Zint_5$
  }
  \label{fig:doubletripleloop}
\end{figure}

\section{Proof of Theorem \ref{thm:allfree}}\label{sec:proof-allfree}

The following proposition is required to prove the theorem.

\begin{proposition}\label{prop:freefree}
  If $H_q(X_n, X_m; \Zint)$
  is free,
  $\coker(\phi_m^n: H_q(X_m; \Zint) \to H_q(X_n; \Zint))$ is also free.
\end{proposition}

\begin{proof}
  We have the following long exact sequence~\cite[Theorem 2.16, pp. 117]{AT}  for the pair $(X_n, X_m)$:
  \begin{align}
    \cdots \to H_q(X_m; \Zint) \xrightarrow{\phi_m^n} H_q(X_n; \Zint)
    \xrightarrow{\psi_m^n} H_q(X_n, X_m; \Zint) \to \cdots,
  \end{align}
  where $\psi_m^n$ is induced by canonical projection.
  Therefore, we have the following relationship between
  $\coker(\phi_m^n)$ and $H_q(X_n, X_m; \Zint)$.
  \begin{align}
    \coker(\phi_m^n) = H_q(X_n; \Zint) / \im \phi_m^n = H_q(X_n; \Zint) / \ker \psi_m^n
    \simeq \im \psi_m^n \subset H_q(X_n, X_m; \Zint)
  \end{align}
  
  To complete the proof, we show that $\im \psi_m^n$ is free,  and this derives from the
  following well-known theorem.
  \begin{theoremext}\label{thmext:submodule}
    Any sub-module of a free $\Zint$-module is also free.
  \end{theoremext}
  
\end{proof}

From the assumption of the theorem,  $H_q(X_m; \Zint) = H_q(X_m, X_0; \Zint)$ is free for all $m$. Hence,
$\phi_m^n: H_q(X_m; \Zint) \to H_q(X_n; \Zint)$ is a homomorphism between two finitely generated
free $\Zint$-modules and the map has a Smith normal form (SNF).
That is, by taking an appropriate basis, $\phi_m^n$ can
be represented by the following $\Zint$ matrix:
\begin{align}\label{eq:snf}
\begin{pmatrix}
  \alpha_1 & 0        & 0      & \cdots & 0        &  \\
  0        & \alpha_2 & 0      & \cdots & 0        &  \\
  0        & 0        & \ddots &        &          &\ \   O \ \ \\
  \vdots   & \vdots   &        & \ddots &          &  \\
  0        & 0        &        &        & \alpha_K & \\
  \\
  &          &    O   &        &          & \ \   O \ \ \\
  \\
\end{pmatrix},
\end{align}
where $0 < \alpha_k \in \Zint$ and $\alpha_k \mid \alpha_{k+1}$ for any $k$. Then from Theorem~\ref{thm:ucm}, the following relationship holds:
\begin{equation}
    \begin{aligned}
    \beta_m^n(\Bbbk) & = \rank (\phi_m^n: H_q(X_m; \Bbbk) \to H_q(X_n; \Bbbk)) \\
    & = \rank (\phi_m^n \otimes \mathrm{id}_\Bbbk: H_q(X_m; \Zint) \otimes \Bbbk \to H_q(X_n; \Zint) \otimes \Bbbk).
    \end{aligned}\label{eq:beta-mn}
\end{equation}
From \eqref{eq:snf} and \eqref{eq:beta-mn}, we know that
$\beta_m^n(\Bbbk)$ is independent of the choice of $\Bbbk$
if and only if $\alpha_1 = \cdots = \alpha_K = 1$.
From SNF, we also have the following:
\begin{align*}
  \coker \phi_m^n \simeq &\bigoplus_{k=k_0}^K(\Zint / \alpha_k\Zint) \oplus \Zint^{L - K}, \\
  \text{where } & k_0 = 1 + \max\{i \mid \alpha_i = 1\}, \\
  & L = \rank H_q(X_n; \Zint).
\end{align*}
This means that $\alpha_1 = \cdots = \alpha_K = 1$ if and only if
$\coker(\phi_m^n)$ is free and the condition is shown
from Prop~\ref{prop:freefree} and the assumption of the theorem.

\subsection{Proof of Corollary~\ref{cor:M_munus_1}}\label{subsec:M_minus_1}
Standard homology theory~\cite[Corollary 3.46, pp. 256]{AT} 
shows that $H_{M-1}(X_n; \Zint)$ and 
$H_{M-2}(X_n; \Zint)$ are free under the condition of this corollary.
The above corollary is shown by using Alexander duality~\cite[Corollary 3.45, pp. 255]{AT} .
$H_{M-1}(X_n, X_m; \Zint)$ is also free since this relative homology group
is isomorphic to a 0th relative cohomology group because of Alexander duality.
Therefore, from Theorem~\ref{thm:allfree}, this can be applicable to the filtration.

\section{Proof of Theorem~\ref{thm:allfree2}}\label{sec:proof-allfree2}

The proof is similar to that of Theorem~\ref{thm:allfree}, but slightly more complex.
We prepare the following proposition.
\begin{proposition}\label{prop:quotfree}
  $H_q(X_n, X_m; \Zint)$ is free if
  $\coker(\phi_m^n: H_q(X_m; \Zint) \to H_q(X_n; \Zint))$ and
  $H_{q-1}(X_m; \Zint)$ are free.
\end{proposition}

\begin{proof}
  From the long exact sequence for the pair $(X_n, X_m)$,
  \begin{equation}
    \begin{aligned}
      \cdots \to H_q(X_m; \Zint) \xrightarrow{\phi_m^n} H_q(X_n; \Zint)
      \xrightarrow{\psi_m^n} H_q(X_n, X_m; \Zint) \xrightarrow{\partial}
      H_{q-1}(X_m; \Zint) \to \cdots,
    \end{aligned}
  \end{equation}
  we have the
  following facts:
  \begin{equation}
    \begin{aligned}
      \coker(\phi_m^n) & \simeq \im \psi_m^n \subset H_q(X_n, X_m; \Zint), \\
      \im \partial & \simeq H_q(X_n, X_m; \Zint) / \im \psi_m^n .
    \end{aligned}
  \end{equation}
  $\im \partial$ is free since $H_{q-1}(X_m)$ is free.
  We complete the proof
  by the following theorem from standard algebra.
  \begin{theoremext}
    Let $M$ be a module over $\Zint$ and $N$ be a sub-module of $M$.
    $M$ is finitely generated and free if $N$ and $M/N$ are
    finitely generated and free.
  \end{theoremext}
\end{proof}

\subsection{Proof of Theorem~\ref{thm:allfree2}}
We assume that $D_q(\X; \Bbbk)$ is independent of the choice of
$\Bbbk$. Then from Lemma~\ref{lem:pbn},
$\beta_m^n(\Bbbk)$ is independent of $\Bbbk$ for any $m$ and $n$.
Especially, for any $n$ and $q$,
$\beta_n^n(\Bbbk) = \dim H_q(X_n; \Bbbk)$ is independent of
$\Bbbk$ and therefore $H_q(X_n; \Zint)$ is free
due to the universal coefficient theorem
since $H_{q-1}(X_n; \Zint)$ is free.
Then $\phi_m^n$ has SNF and $\coker \phi_m^n$ is free
because of the discussion in the proof of Theorem~\ref{thm:allfree}.
From the above fact and Proposition~\ref{prop:quotfree}, we conclude that
$H_q(X_n, X_m; \Zint)$ is free.

\subsection{Proof of Corollary~\ref{cor:allfree3}}

From Theorem~\ref{thm:allfree}, it is straightforward to show that 
$D_q(\X; \Bbbk)$ is independent of the choice of $\Bbbk$ for all $q = 0, \ldots, M$ if
  $H_q(X_n, X_m; \Zint)$ are free for any $0\leq m < n \in T$, and $q = 0, \ldots, M$.

We can show the converse by induction on $q$.
For $q = 0$, it is trivial that $H_0(X_n, X_m; \Zint)$ is free and
the induction process proceeds by using Theorem~\ref{thm:allfree2}.

\section{Algorithm to determine the dependency of $D_q(\X; \Bbbk)$ on $\Bbbk$}\label{sec:algorithm}

In this section, we explore an algorithm to judge the existence of non-zero
torsion and prove Theorem~\ref{thm:allfreealg}. See Algorithm~\ref{alg:allfree}. Now we prove the following facts.
\begin{itemize}
\item If the algorithm returns ``independent'',
  the given filtration satisfies the condition of Corollary~\ref{cor:allfree3}.
  Therefore, $D_q(\X; \Bbbk)$ is independent of the choice of $\Bbbk$.
\item If the algorithm returns ``dependent'',
  the given filtration does not satisfy the condition of Corollary~\ref{cor:allfree3}.
  Therefore, $D_q(\X; \Bbbk)$ depends on the choice of $\Bbbk$.
\end{itemize}

\begin{algorithm}[ht]
  \caption{Algorithm to determine the dependency of $D_q(\X; \Bbbk)$ on $\Bbbk$}\label{alg:allfree}
  \begin{algorithmic}
    \State let $B$ be the matrix representation of the boundary operator 
    \For{$j=1, \ldots, N$} \Comment (OUTERLOOP)
      \While{there exists $i < j$ with $\low_B(i) = \low_B(j) \not = -\infty$} \Comment (INNERLOOP)
         \State let $s = - B_{\low_B(j), j} / B_{\low_B(i), i}$ \Comment (A)
         \State add $s \times (\textrm{column}\  i)$ to column $j$ of $B$ \Comment left-to-right reduction
      \EndWhile
      \If{$L(j) \not = -\infty$ and $B_{L_{B}(j), j} \not \in \{\pm 1\}$} \Comment (B)
        \State print $|B_{L_B(j), j}|$
        \State \Return ``dependent''
      \EndIf
    \EndFor
    \State \Return ``independent''
  \end{algorithmic}
\end{algorithm}

In algorithm \ref{alg:allfree}, $L_B$ means \eqref{eq:LB}.
We remark that $B_{L_{B}(i), i}$ in this algorithm is always $\pm 1$ at (A), so
the division at (A) always applies. This is because the condition is
checked at (B).

For the proof, we use ideas presented in the Appendix~\ref{sec:proof-alg1}.
We use the following notation. 

\begin{notation}\label{notation:allfree}
  \ 
  \begin{itemize}
  \item $R$ is $\Zint$ or a field
  \item $C(X_k; R) := \bigoplus_{q=0}^{\dim(\X)} C_q(X_k; R)$
  \item $\partial_k: C(X_k; R) \to C(X_k; R) = \bigoplus_{q=0}^{\dim(\X)} (\partial_k^{(q)}:C_q(X_k; R) \to C_{q-1}(X_k; R))$ 
  \item $Z(X_k; R):= \ker \partial_k  = \bigoplus_{q=0}^{\dim(\X)} \ker \partial_k^{(q)}$
  \item $B(X_k; R):=\mathrm{im} \partial_{k} = \bigoplus_{q=0}^{\dim(\X)} \mathrm{im} \partial_{k}^{(q)}$
  \item $H(X_k; R):= Z(X_k; R) / B(X_k; R) = \bigoplus_{q=0}^{\dim(\X)} H_q(X_k; R)$ 
  \item $[\cdot]_k$ is a homology class in $H(X_k; R)$
  \item $B_0$ is the boundary matrix of $\partial = \partial_N$
    with respect to the basis $\{\sigma_1, \ldots, \sigma_N\}$
  \end{itemize}
\end{notation}

We check whether
$H(X_n, X_m; \Zint) = \bigoplus_{q=0}^{M} H_q(X_n, X_m; \Zint)$
has a non-zero torsion subgroup $T(H(X_n, X_m; \Zint))=  \bigoplus_{q=0}^{M} T(H_q(X_n, X_m; \Zint))$ for every $m < n$.
We assume that $\dim(\X) \leq M + 1$ by removing all simplices whose dimensions
exceed $M + 1$.

First, we consider the following fact.
\begin{fact}\label{fact:algfree-finite-steps}
  Algorithm~\ref{alg:allfree} always terminates in finite steps.
\end{fact}
Fact~\ref{fact:algfree-finite-steps} can be easily shown since,
in the while loop (INNERLOOP), $L_B(j)$ is strictly monotonically decreasing
and finally $L_B(j)$ becomes $-\infty$ or distinct from $\{L_B(i) \mid i < j\}$.

The following lemma is important for the proof. The lemma is proved at the end of this section.

\begin{lemma}\label{lem:outerloop}
  Let $n$ be an integer satisfying $ 1 \leq n \leq N+1$.
  If (OUTERLOOP) in Algorithm~\ref{alg:allfree} is done for $j = 1, \ldots, n-1$,
  we can show the existence of 
  a basis of $C(X_{n-1})$,
  $\{\tilde{\sigma}_1, \ldots, \tilde{\sigma}_{n-1}\}$, and
  a decomposition of $\{1, \ldots, n-1\}$, $D_n \sqcup D_n' \sqcup E_n$,
  such that the following conditions hold.
  \begin{enumerate}[(a)]
  \item\label{enum:basis} $\{\tilde{\sigma}_1,\ldots, \tilde{\sigma}_k\}$ is a basis of $C(X_k; \Zint)$
    for any $1 \leq k < n$.
  \item $\partial \tilde{\sigma}_j \not = 0$ for $j \in D_n$.
  \item $L_{\hat{B}}$ is a bijection from $D_n$ to $D'_n$ and
    $\partial \tilde{\sigma}_j = \tilde{\sigma}_{L_{\hat{B}}(j)}$ for any $j \in D_n$.
  \item $\partial \tilde{\sigma}_i = 0$ for $i \in D'_n \sqcup E_n$.
  \end{enumerate}
  
  Additionally, if condition (B) in Algorithm~\ref{alg:allfree} is true at $j = n$
  for $ 1 \leq n \leq N$,
  $L_{\hat{B}}(n)$ is not $-\infty$ and
  there exists $\hat{\sigma}_n \in C(X_n; \Zint)$ such that the
  following conditions hold:
  \begin{equation}\label{eq:reminder}
    \begin{aligned}
      C(X_n; \Zint) & = C(X_{n-1}; \Zint) \oplus \left< \hat{\sigma}_n \right>, \\
      \partial \hat{\sigma}_n & = \hat{B}_{L_{\hat{B}}(n), n} \tilde{\sigma}_{L_{\hat{B}}(n)} +
      \sum_{1 \leq i < L_{\hat{B}}(n)} \beta_i \tilde{\sigma}_i,  \\
      \partial \sigma_{L_{\hat{B}}(n)} &= 0.
    \end{aligned}
  \end{equation}
  From the above conditions we can show $L_{\hat{B}}(n) \in E_n$.
\end{lemma}

First we prove $T(H_q(X_n, X_m; \Zint)) = 0$ for every $m < n$
when the algorithm returns ``independent''. When the algorithm returns ``independent'',
we can apply Lemma~\ref{lem:outerloop} for $n=N+1$.
Therefore, we can find a basis of
$C(X_N; \Zint)$, $\{\tilde{\sigma}_1, \ldots, \tilde{\sigma}_N\}$, satisfying conditions
(a), (b), (c), and (d) in Lemma~\ref{lem:outerloop}. We can explicitly write the bases of
$C(X_m; \Zint)$ and $C(X_n, X_m; \Zint) $ as follows:
\begin{equation}
  \begin{aligned}
    &\{\tilde{\sigma}_{1}, \ldots, \tilde{\sigma}_{m} \}
    \text{ is a basis of } C(X_m; \Zint), \\
    &\{\tilde{\sigma}_{m+1} + C(X_m;\Zint), \ldots, \tilde{\sigma}_{n} + C(X_m;\Zint) \}
    \text{ is a basis of } C(X_n, X_m; \Zint). \\
  \end{aligned}
\end{equation}
Let $\partial_{n,m}: C(X_n, X_m; \Zint) \to C(X_n, X_m; \Zint) $ be
the boundary operator on relative chain complexes. 
From Condition~\ref{cond:phbasis} (b), (c), and (d), 
$\ker \partial_{n,m}$ and $\im \partial_{n,m}$ are both $\Zint$ free modules
and the bases are
\begin{equation}
  \begin{aligned}
    &\{\tilde{\sigma}_k + C(X_m; \Zint) \mid  m < k \leq n, k \in D' \sqcup E\} \\
    \sqcup &\{\tilde{\sigma}_k + C(X_m; \Zint) \mid  m < k \leq n, k \in D, L_{\hat{B}}(k) < m\},
  \end{aligned}
\end{equation}
and
\begin{equation}
  \begin{aligned}
    &\{\tilde{\sigma}_{L(k)} + C(X_m; \Zint) \mid  m < k \leq n, k \in D, m < L_{\hat{B}}(k)\} \\
    =&\{\tilde{\sigma}_{k} + C(X_m; \Zint) \mid  m < k \leq n, k \in D', L_{\hat{B}}^{-1}(k) \leq n \}.
  \end{aligned}
\end{equation}

Therefore, the basis of $H(X_n, X_m; \Zint) = \ker \partial_{n,m} / \im \partial_{n,m}$
can be written as follows:
\begin{equation}
  \begin{aligned}
    & \{[\tilde{\sigma}_k + C(X_m; \Zint)] \mid
    m < k \leq n, k \in E \} \\
    \sqcup & \{[\tilde{\sigma}_k + C(X_m; \Zint)] \mid
    m < k \leq n, k \in D', n < L_{\hat{B}}^{-1}(k) \},
  \end{aligned}
\end{equation}
where $[z + C(X_m; \Zint)]$ is the homology class of $z + C(X_m; \Zint)$
in $H(X_n, X_m; \Zint)$.
Therefore, $H(X_n, X_m; \Zint)$ is a free $\Zint$-module and
we complete the proof for the ``independent'' case.

Next we show that there is a pair $(m, n)$ such that
$T(H(X_n, X_m; \Zint)) \not = 0$ if the algorithm returns ``dependent''.
In that case, condition (B) in the algorithm is true for one $j$, so
let $n$ be that $j$. We can apply Lemma~\ref{lem:outerloop} for that $n$.
In this case the conditions in \eqref{eq:reminder} hold.

Now let $m := L_{\hat{B}}(n) - 1$.
From (a)-(d) and \eqref{eq:reminder},
the bases of $\ker \partial_{n,m}$ and $\im \partial_{n,m}$ can be explicitly
written as follows:
\begin{equation}
  \begin{aligned}
    &\{\tilde{\sigma}_k + C(X_m; \Zint) \mid  m < k < n, k \in D'_n \sqcup E_n\} \\
    \sqcup &\{\tilde{\sigma}_k + C(X_m; \Zint) \mid  m < k < n, k \in D_n, L_{\hat{B}}(k) < m\},
  \end{aligned}
\end{equation}
and
\begin{equation}
  \begin{aligned}
    &\{\tilde{\sigma}_{L_{\hat{B}}(k)} + C(X_m; \Zint) \mid  m < k < n, k \in D_n, m < L_{\hat{B}}(k)\} \\
    \sqcup & \{ p \tilde{\sigma}_{m+1} + C(X_m; \Zint) \},
  \end{aligned}
\end{equation}
where $p = |\hat{B}_{m+1, n}|$. Using $m + 1 \in E_n$, finally we have
\begin{equation}
  \begin{aligned}
    H(X_n, X_m; \Zint) =&
    \left< [\tilde{\sigma}_{k} + C(X_m; \Zint)] \mid
      m < k < n, k \in D'_n, L_{\hat{B}}^{-1}(k) \leq n \right> \\
    \oplus & \left< [\tilde{\sigma}_{k} + C(X_m; \Zint)] \mid
      m < k < n, k \in E_n\backslash\{ m + 1\}  \right> \\
    \oplus& \left< [\tilde{\sigma}_{m+1} + C(X_m; \Zint)] \right>
  \end{aligned}
\end{equation}
and
\begin{equation}
  \left< [\tilde{\sigma}_{m+1} + C(X_m; \Zint)] \right> \simeq \Zint_p.
\end{equation}
The proof for the ``dependent'' case is completed.

Here, when the algorithm returns ``dependent'', the number $p$ is displayed. This facilitates understanding the dependency
of $D(\X; \Zint_{p'})$ to a prime $p'$ which is a divisor of $p$.

\subsection{Proof of Lemma~\ref{lem:outerloop}}
\begin{proof}
  Let $\hat{B}$ be matrix $B$ in Algorithm~\ref{alg:allfree}
  when (OUTERLOOP) is executed for $j = 1, \ldots, n-1$.
  Since the left-to-right reduction in Algorithm~\ref{alg:allfree} is equivalent to the
  multiplication of the following invertible matrix $R_{ij}(s)$ from right,
  \begin{equation}
    R_{ij}(s) = 
    \begin{blockarray}{cccccccc}
      &&&&j&&& \\
      \begin{block}{(ccccccc)c}
        1 &        &   &        &   &        &&\\ 
        & \ddots &   &        &   &        &&\\ 
        &        & 1 &        & s &        && i\\ 
        &        &   & \ddots &   &        &&\\ 
        &        &   &        & 1 &        &&\\ 
        &        &   &        &   & \ddots &&\\ 
        &        &   &        &   &        & 1 &\\
      \end{block}
    \end{blockarray},
  \end{equation}
  $\hat{B}$ can be written as
  \begin{equation}
    \hat{B} = B_0 U,
  \end{equation}
  where $U$ is an upper triangular matrix whose
  diagonal elements are all 1. Since $U^{-1}$ is also an upper triangular
  matrix whose diagonal elements are 1,
  we can show the following fact by elemental matrix calculus.

  \begin{fact}
    For all $1 \leq j \leq N$, 
    $L_{\hat{B}}(j) = L_{U^{-1}\hat{B}}(j)$.
  \end{fact}

  Now we define $\tilde{B}$ and $L(j)$ as follows for the proof.
  \begin{equation}
    \begin{aligned}
      \tilde{B} &= U^{-1}\hat{B} = U^{-1}B_0U, \\
      L(j) & = L_{\hat{B}}(j) = L_{U^{-1}\hat{B}}(j) = L_{\tilde{B}}(j).
    \end{aligned}
  \end{equation}

  We consider the matrix $\tilde{B} = U^{-1}B_0U$. Let
  $\{\hat{\sigma}_1, \ldots, \hat{\sigma}_N\}$ be the basis of 
  $C(X_N; \Zint)$ given by the change of coordinate matrix $U$. Since $U$ is
  upper triangular, the following relation holds for every $1 \leq k < n$.
  \begin{equation}\label{eq:same-space}
    C(X_k; \Zint) = 
    \left< \sigma_1, \ldots, \sigma_k \right> =
    \left< \hat{\sigma}_1, \ldots, \hat{\sigma}_k \right>.
  \end{equation}
  From the terminating condition of the while loop (INNERLOOP) in Algorithm~\ref{alg:pd}, we also
  have the following facts.
  \begin{fact}\label{fact:distinct}
    If $1 \leq i, j < n$ satisfy $i \not = j$, $L(i) \not = -\infty$, and
    $L(j) \not = -\infty$, then
    $L(i) \not = L(j)$.
  \end{fact}

  Now we consider a column of $\tilde{B}$ which is non-zero. Let $j$ be the index of
  the column. From the definition of $L$, we have the following relationship:
  \begin{equation}
    \begin{aligned}
      L(j) & \not= -\infty, \\
      \tilde{B}_{L(j),j} & \not = 0, \\
      \tilde{B}_{k,j} & = 0 \text{ for any } k > L(j).
    \end{aligned}
  \end{equation}
  Hence, $\partial \hat{\sigma}_j$ can be written as follows:
  \begin{equation}\label{eq:boundary-op-hat}
    \partial \hat{\sigma}_j = \tilde{B}_{L(j), j} \hat{\sigma}_{L(j)} +
    \sum_{1 \leq i < L(j)} \tilde{B}_{ij} \hat{\sigma}_i.
  \end{equation}
  Now we show the following claim.
  \begin{claim}\label{claim:zero_lowj}
    The $\low(j)$-th column of $\tilde{B}$ is zero.
  \end{claim}

  We prove the claim by contradiction. We assume that
  the $\low(j)$-th column of $\tilde{B}$ is non-zero.
  By applying $\partial$ to \eqref{eq:boundary-op-hat}, we have
  \begin{equation}\label{eq:boundary-0}
    \begin{aligned}
      0 &= \tilde{B}_{L(j), j} \partial \hat{\sigma}_{L(j)} +
      \sum_{1 \leq i < L(j)} \tilde{B}_{ij} \partial \hat{\sigma}_i.
    \end{aligned}
  \end{equation}
  Let $I$ be
  \begin{equation}
    I = \{i \mid 1 \leq i \leq L(j), \tilde{B}_{ij} \not = 0 \text{ and }
    \partial \hat{\sigma}_i \not = 0\}.
  \end{equation}

  Clearly, from \eqref{eq:boundary-0}, we have
  \begin{equation}\label{eq:boundary-nonzero}
    \begin{aligned}
      0 =& \sum_{i \in I} \tilde{B}_{ij} \partial \hat{\sigma}_i 
      = \sum_{i \in I} \tilde{B}_{ij}\left(
        \tilde{B}_{L(i), i} \hat{\sigma}_{L(i)}
        + \sum_{1 \leq k < L(i)} \tilde{B}_{ki} \hat{\sigma}_{k}
      \right).
    \end{aligned}  
  \end{equation}
  From the assumption that the $\low(j)$-th column of $\tilde{B}$ is non-zero,
  $\partial \hat{\sigma}_{L(j)}$ is non-zero and we have $j \in I$, so
  $I$ is non-empty.
  Now we consider $L(I) = \{L(i) \mid i \in I\}$.
  When we consider the maximum of $L(I)$,
  the index $i_0$ attaining the maximum is unique due to Fact~\ref{fact:distinct}.
  Therefore,
  in \eqref{eq:boundary-nonzero},
  a term of $\hat{\sigma}_{L(i_0)}$ appears only once and the coefficient of
  the term is $\tilde{B}_{i_0, j}\tilde{B}_{L(i_0),i_0}$. This value is non-zero
  because of the definition of $I$ and this contradicts \eqref{eq:boundary-nonzero}.

  We define $D_n, D_n'$, and $E_n$ as follows:
  \begin{equation}
    \begin{aligned}
      D_n =& \{ 1 \leq j < n  \mid \text{the $j$th column of $\tilde{B}$ is non-zero}\}, \\
      D_n' =& \{L(j) \mid j \in D_n\}, \\
      E_n =& \{1 \leq j < n \mid \text{the $j$th column of $\tilde{B}$ is zero and
        $L(i) \not = j$ for any $i \in D_n$ }\}.
    \end{aligned}  
  \end{equation}
  From the above claim, $D_n$ and $D_n'$ have no common element and hence
  $D_n, D_n',$ and $E_n$ are the decomposition of indices $\{1, \ldots, n - 1\}$.
  The map $L$ from $D$ to $D'$ is bijective because of Fact~\ref{fact:distinct}.
  Chains $\{\tilde{\sigma}_1, \ldots, \tilde{\sigma}_{n-1} \}$ are defined as follows:
  \begin{equation}\label{eq:def-tilde-sigma}
    \tilde{\sigma}_k = \left\{
      \begin{array}{ll}
        \hat{\sigma}_k
        & \text{if $k \in D_n \sqcup E_n$}, \\
        \tilde{B}_{kj} \hat{\sigma}_{k} +
        \sum_{1 \leq i < k} \tilde{B}_{ij} \hat{\sigma}_i
        & \text{if $k \in D_n'$ and $j = L^{-1}(k)$ }.p
      \end{array}
    \right. 
  \end{equation}
  The set of chains $\{\tilde{\sigma}_1, \ldots, \tilde{\sigma}_{n-1}\}$ satisfies
  condition (a) in Lemma~\ref{lem:outerloop}
  due to \eqref{eq:same-space} and \eqref{eq:def-tilde-sigma}.
  It is also straightforward to show conditions (b), (c), and (d) from the construction of
  the decomposition and the chains.

  It is easy to prove the conditions in \eqref{eq:reminder} from the above discussion.
  When condition (B) in the algorithm is true for $j = n$, 
  we can show Fact~\ref{fact:distinct} for $ 1 \leq i, j \leq n$
  since (INNERLOOP) already terminates for $j = n$.

\end{proof}

\section{Algorithm implementation}\label{sec:implementation}
The judgement algorithm is implemented in HomCloud\footnote{\url{https://www.wpi-aimr.tohoku.ac.jp/hiraoka_labo/homcloud/index.en.html}}. 
The twist algorithm introduced by
\cite{Chen11persistenthomology} is used for faster
computations. The program correctly judges the existence of the torsion
for pointclouds shown in Fig.~\ref{fig:doubletripleloop} (a) and (e).

\subsection{Performance benchmark}

In this section, we explore the performance of Algorithm~\ref{alg:allfree}.
We compare the program implemented in HomCloud
and Phat~\cite{phat}\footnote{\url{https://bitbucket.org/phat-code/phat/}}.
The Phat code is straightforward and efficient.
The input filtration for the performance comparison is 
an alpha filtration constructed from random 5000 and 50000 points in $\R^3$.
Five trials were undertaken and the average computation time is shown.
In Phat, we use twist-algorithm with \verb|bit_tree_pivot_column|, as recommended by
\cite{phat}. The benchmark is executed on a PC with a 1.5 GHz Intel(R) Core(TM) i7-8500Y CPU, 16 GB of memory, and the Debian 10.0 operating system.
Both programs run on a single core.
Results are shown in Table~\ref{tab:performance}.

\begin{table}[htbp]
  \centering
  \begin{tabular}{c|c|c} \hline
    & 5000 points & 50000 points \\ \hline
    Phat &  0.0282 sec & 0.446 sec \\ \hline
    HomCloud & 0.0323 sec & 0.550 sec \\ \hline
  \end{tabular}
  \caption{Performance benchmark results}
  \label{tab:performance}
\end{table}

According to the benchmark, our new program is c.~$\times$1.20 slower than Phat.
Phat uses $\Zint_2$ as a coefficient field and implements
fast arithmetic operations by using bit-wise operations. The technique likely renders Phat faster and we conclude that the performance of 
our program is roughly as efficient as Phat.

\section{Probability of torsion appearance}\label{sec:probability}
Here we measured the probability of the appearance of torsions
of random filtrations by a numerical experiment.
Corollary~\ref{cor:M_0} and Corollary~\ref{cor:M_munus_1} already ensure
the independence of persistence diagrams from $\Bbbk$
for a filtration embedded in $\R^2$. Therefore, we started from filtrations
in $\R^3$.

We generated a random filtration from a pointcloud sampled from 
a Poisson point process in $[0, 1]^3$.
The average number of points is 1000. 
Thus, a random number $k$ is sampled from the Poisson
distribution whose parameter is 1000 and $k$ points are uniformly randomly
sampled in $[0, 1]^3$. An alpha filtration was computed
from the generated pointcloud and the condition was judged by 
HomCloud. Here, 10000 trials were carried out.
Only one filtration had non-trivial torsion; thus, 9999 filtrations had trivial torsion\footnote{
 Run-time errors occurred two times in these 10000 trials.
  When an error occurred, we disposed of the input
  data and retried random sampling. The cause of the errors is probably
  the violation of the general position condition of
  the randomly generated pointcloud.
}. In sum, it can be stated that a filtration with non-trivial torsion is possible, but very rare.
This numerical experiment suggests that there is some mathematical
mechanism explaining why a random filtration with non-trivial torsion is quite
rare. Exploring this further here is beyond the scope of the current paper.

In contrast \cite{doi:10.1080/10586458.2018.1473821} experimentally showed that torsion subgroups often
appeared in random $d$-complex $Y \sim Y_d(n, p)$, introduced by \cite{Linial*2006}. 
We apply our algorithm to random filtrations used in that paper. Let $\bar{Y}(n)$ be a simplex on $n$ vertices and let $Y_0$
be the $(d-1)$-skeleton of $\bar{Y}(n)$.
$Y_k$ for $k=1, \ldots, m$ is randomly generated by adding a $d$-simplex to $Y_{k-1}$. The $d$-simplex
is uniformly randomly sampled from all $d$-simplices in $\bar{Y}(n) \backslash Y_{k-1}$. We apply the algorithm to the
filtration
$Y_0 \subset Y_1 \subset \cdots \subset Y_m$. We used $d=2, n=75, m=5000$. The number of random flirtations
was 10000. In the experiment we found that all 10000 random filtrations have non-trivial torsion.

The above two experiments are contrasting. We expect that the difference emanates  from the
dimension of the space. In the first experiment, a filtration is embedded in $\R^3$ and
in the second experiment $\bar{Y}(75)$ can be embedded in $\R^{74}$. The experiments
suggest that a random filtration embedded in a higher dimensional space has more non-trivial torsion subgroups
in the relative homology groups than a filtration embedded in a lower dimensional space.

To further investigate the relationship between the dimension of the space and the non-trivial torsion, we numerically experimented on random Vietoris-Rips filtrations. Vietoris-Rips filtrations were used since it is difficult to construct an alpha filtration of a pointcloud in a high-dimensional space. In one trial of the experiment, 1000 points are uniformly randomly sampled in $\R^n$
and we judged the existence or non-existence of non-trivial torsion subgroup in the 1st persistent homology of the Vietoris-Rips filtration of the pointcloud using our algorithm. To reduce the cost of computation we take the threshold of maximum edge length. The threshold is statistically determined to averagely include 166500 edges (1/6 edges of all edges of 999-simplex) in the filtration. We determined the threshold rule by considering the limitation of our computer resource. We performed 1000 trials for each $n$ and counted frequencies for the non-trivial torsion subgroup. 

Table \ref{tab:dim_freq} and Fig.~\ref{fig:dim_freq} show the frequencies for $n = 3, 4, \ldots, 40$. We also experimented with larger thresholds to examine the validity of the threshold rule and the results were consistent with the 1/6 rule experiments. The results show that the frequency appears to monotonically increase with $n$ and the speed of increase becomes slower as $n$ increases.   

\begin{table}[hbtp]
    \centering
    \begin{tabular}{c|c|c|c|c|c|c|c|c|c|c} \hline
        $n$ & & 2 & 3 & 4 & 5 & 6 & 7 & 8 & 9 & 10 \\ \hline
        frequency & & 0 &  0 &  4&  9& 17& 14& 25& 25& 25 \\ \hline \hline
        & 11 & 12 & 13 & 14 & 15 & 16 & 17 & 18 & 19 & 20  \\ \hline
        & 29 & 36 & 31 & 30 & 43 & 36 & 32 & 37 & 32 & 30 \\ \hline \hline
        & 21 & 22 & 23 & 24 & 25 & 26 & 27 & 28 & 29 & 30  \\ \hline
        & 44 & 47 & 42 & 53 & 39 & 46 & 46 & 47 & 42 & 35 \\ \hline \hline
        & 31 & 32 & 33 & 34 & 35 & 36 & 37 & 38 & 39 & 40  \\ \hline
        & 47 & 45 & 43 & 45 & 49 & 50 & 44 & 30 & 39 & 58 \\ \hline 
    \end{tabular}
   \caption{dimension $n$ vs frequency}\label{tab:dim_freq}
    \label{tab:my_label}
\end{table}

\begin{figure}[hbtp]
    \centering
    \includegraphics[width=0.5\hsize]{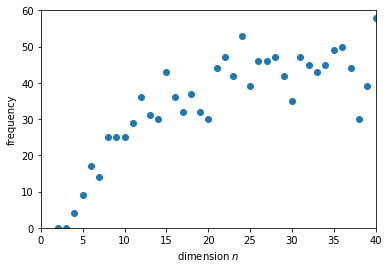}
    \caption{Scatter plot of Table \ref{tab:dim_freq}}
    \label{fig:dim_freq}
\end{figure}

Overall, the experiment also suggests that we do not need to 
be concerned about the coefficient field in most cases if the space is $\R^3$. Thus, if there are concerns about the field choice problem in future research contexts, our proposed algorithm would be helpful.

\section{Proof of Theorem~\ref{thm:r-wasserstein}}\label{sec:convex}

In this section, we assume the following conditions.
\begin{cond}
There exists $0 = r_0 < r_1 < \ldots < r_N < r_{N+1} = \infty$ such that
$r_k \leq r < r_{k+1}$ implies $X_r = X_{r_k}$.
\end{cond}
This means that the filtration is assumed to be right-continuous.
This condition is not essential and we can prove the theorem if
the filtration is left-continuous. We assume right-continuous for expository purposes.
From this condition and $H_q( \cup_t X_t) = 0$, all birth-death pairs can be
written as $(r_k, r_\ell)$ for $0 \leq k < \ell \leq N$.
Hence, using persistent Betti numbers, we have the following equation:
\begin{equation}
    \begin{aligned}
    & \sum_{(b, d) \in D_q(\X; \R)} f(d - b) \\
    =& \sum_{0 \leq k < \ell \leq N}f(r_\ell - r_k)\left(
    \beta_{r_k}^{r_{\ell-1}}(\Bbbk) - \beta_{r_{k-1}}^{r_{\ell-1}}(\Bbbk)
    - \beta_{r_k}^{r_{\ell}}(\Bbbk) + \beta_{r_{k-1}}^{r_{\ell}}(\Bbbk)
    \right) \\
    =& \sum_{0 \leq k < \ell \leq N} \beta_{r_{k}}^{r_{\ell}}(\Bbbk)(
    f(r_{\ell+1} - r_{k}) - f(r_{\ell+1} - r_{k+1})
    - f(r_{\ell} - r_{k}) + f(r_{\ell} - r_{k+1}) ).
    \end{aligned}\label{eq:f_sum}
\end{equation}

Now we prove the following inequality:
\begin{equation}
f(r_{\ell+1} - r_{k}) - f(r_{\ell+1} - r_{k+1})
    - f(r_{\ell} - r_{k}) + f(r_{\ell} - r_{k+1}) ) \geq 0. \label{eq:f_convex}
\end{equation}
In addition, if $f$ is strictly convex, the left-hand side is
strictly positive.
First we prove \eqref{eq:f_convex} under the condition of
$r_{\ell+1} - r_{k+1} \geq r_{\ell} - r_{k}$. In this case,
\begin{equation}
    \begin{aligned}
    & f(r_{\ell+1} - r_{k}) - f(r_{\ell+1} - r_{k+1})
    - f(r_{\ell} - r_{k}) + f(r_{\ell} - r_{k+1}) \\
    = & \left( f(r_{\ell+1} - r_{k}) - f(r_{\ell+1} - r_{k+1}) \right)
    - \left( f(r_{\ell} - r_{k}) - f(r_{\ell} - r_{k+1}) \right) \\
    =& f'(\zeta_1) (r_{k+1} - r_{k}) - f'(\zeta_2)(r_{k+1} - r_{k}) \\
    =& (r_{k+1} - r_{k})(f'(\zeta_1) - f'(\zeta_2)),
    \end{aligned}
\end{equation}
where $r_{\ell+1} - r_{k+1} \leq \zeta_1 \leq r_{\ell+1} - r_{k}$ and
$r_{\ell} - r_{k+1} \leq \zeta_2 \leq r_{\ell} - r_{k}$. Here, 
from the assumption of $r_{\ell+1} - r_{k+1} \geq r_{\ell} - r_{k}$ and
the convexity of $f$, we have $f'(\zeta_1) - f'(\zeta_2) \geq 0$ and
the inequality \eqref{eq:f_convex}. The strict positivity from
strict convexity is trivial. When $r_{\ell+1} - r_{k+1} \leq r_{\ell} - r_{k}$,
we can prove the inequality in a similar way by exchanging the role
of $f(r_{\ell+1} - r_{k+1})$ and $f(r_{\ell} - r_{k})$ in the foregoing.

From the discussion in Section~\ref{sec:proof-allfree}, we have
\begin{equation}
\beta_{r_{k}}^{r_{\ell}}(\R) \geq \beta_{r_{k}}^{r_{\ell}}(\Zint_p),
\label{eq:beta_ineq_1}
\end{equation}
for any $p$, $k$, and $\ell$. Furthermore,
 if $D_q(\X; \R) \not = D_q(\X; \Zint_p)$, there exists $k < \ell$ such that
\begin{equation}
\beta_{r_{k}}^{r_{\ell}}(\R) > \beta_{r_{k}}^{r_{\ell}}(\Zint_p)    
\label{eq:beta_ineq_2}
\end{equation}
holds. From \eqref{eq:f_sum}, \eqref{eq:f_convex}, \eqref{eq:beta_ineq_1}, 
and \eqref{eq:beta_ineq_2}, we complete the proof of the theorem.

\section{Conclusions}\label{sec:conclusion}

In this paper, we focus on mathematical phenomena concerning the change of
the coefficient field in persistent homology. We show that the torsion
of relative homology groups $H_q(X_n, X_m; \Zint)$ plays an essential role
for the phenomena. We also propose an algorithm to judge the independence
of the field change. The algorithm is implemented in software, HomCloud.

Using the algorithm, we conduct experiments which suggest that the probability of persistence
diagrams changing as a result of field changes is not zero, but very low 
for random pointclouds in $\R^3$. Thus, we do not need to be particularly concerned about
the choice of the field in most practical persistent homology contexts
if we approach persistence diagrams in statistical terms. To assuage researchers' future concerns about this issue, the torsion condition can be checked by the algorithm.

Of course, where torsion structures are important, such as 
Klein bottles or M\"obius strip, the choice of the coefficient field is 
important. Based on the results of the numerical experiment on $\bar{Y}(75)$ and Vietoris-Rips filtrations,
we also suggest that the choice of the coefficient field is 
important for high dimensional data. 
In such contexts, further study is required into the torsion
on the filtrations.

Further, the results herein suggest that the
``difficulty'' of computation of $D_q(\X; \Bbbk)$ depends on the torsion.
If all torsions are zero, $D_q(\X; \Bbbk)$ for any $\Bbbk$ is computable
by computing $D_q(\X; \Bbbk)$ for only one $\Bbbk$, for example, $\Zint_2$.
If not, to compute $D_q(\X; \Bbbk)$ for many $\Bbbk$ is more onerous.
In that case, we can apply the algorithm in \cite{BJDMC} for faster computation;
however, that algorithm simultaneously computes $D_q(\X; \Bbbk)$ for multiple, but
not for all, $\Bbbk$. This phenomenon is not dissimilar to a theorem in \cite{optimal-Day}. Those authors proved that
the difficulty of computing a kind of
optimization problem on homology algebra depends on the existence of the non-zero torsion
subgroup of the relative homology group.
Integer programming on homology algebra can be solved by linear programming
if the torsion-free condition holds. Integer programming requires much more time
than linear programming in the sense of computational complexity theory.
Of course, our paper and \cite{optimal-Day} concern different problems,
but the results are similar because of the shared focus on the
torsions of relative homology. These results suggest that
the existence of non-trivial torsion subgroups in relative homology
renders the problems of
computational homology more difficult.


\appendix

\section{Algorithm~\ref{alg:pd}}\label{sec:proof-alg1}

Here we prove the structural theorem of persistent homology and
show why Algorithm~\ref{alg:pd} yields the correct decomposition.
Notation and facts from Section~\ref{sec:ph} are used as necessary.

\begin{theoremext}\label{thmext:pdalg}
  Let $\Bbbk$ be a field and 
  assume that a simplicial filtration
  $\X = \emptyset = X_0 \subset \cdots \subset X_N$
  satisfies Condition~\ref{cond:finite}. Then the persistent homology
  \begin{equation}
    H(\X; \Bbbk) = H(X_0; \Bbbk) \to \cdots \to H(X_N; \Bbbk)
  \end{equation}
  has a unique interval decomposition as follows:
  \begin{equation}\label{eq:interval-decomposition}
    \begin{aligned}
      H(\X; \Bbbk) & \simeq \bigoplus_{i=1}^L I(b_m, d_m), \\
      I(b, d) &= 0 \to \cdots \to 0 \to \overset{b}{\check{\Bbbk}}
      \xrightarrow{1} \cdots 
      \xrightarrow{1} \overset{d - 1}{\check{\Bbbk}}
      \to \overset{d}{\check{0}} \to \cdots \to 0, \\
      I(b, \infty) &= 0 \to \cdots \to 0 \to \overset{b}{\check{\Bbbk}}
      \xrightarrow{1} \cdots \xrightarrow{1} \Bbbk,
    \end{aligned}
  \end{equation}
  where $H(X_k; \Bbbk) = \bigoplus_{q=0}^{\dim(\X)} H_q(X_k; \Bbbk)$
  is a direct sum of $q$th PDs,
  $\dim(\X)$ is the maximum dimension of the simplices in $\X$,
  $H(X_k; \Bbbk) \to H(X_{k+1}; \Bbbk)$ is the homology homomorphism 
  induced by the inclusion map $X_k \hookrightarrow X_{k+1}$,
  $b_m \in \{1, \ldots, N\}$, $d_m \in \{2, \ldots, N\} \cup \{\infty\}$,
  $b_m < d_m$,
  and
  $\Bbbk \xrightarrow{1} \Bbbk$ is the identity map on $\Bbbk$.
  Moreover Algorithm~\ref{alg:pd} gives the PD
  as shown in \eqref{eq:pd}.
\end{theoremext}

We use the notation in Notation~\ref{notation:allfree} with $R=\Bbbk$.
To start the proof, we consider the meaning of an
interval indecomposable $I(b, d)$.
\begin{itemize}
\item A new homology generator is born at $H(X_{b}; \Bbbk)$. This means that
  a cycle exists which satisfies
  $z \in Z(X_b; \Bbbk) \backslash Z(X_{b-1}; \Bbbk)$.
\item The homology generator persists until $H(X_{d-1}; \Bbbk)$. This means that
  $[z]_k \in H(X_k; \Bbbk)$ is non-zero for any $b \leq k < d$.
\item The homology generator dies at $d$. This means that
  $[z]_d = 0 \in H(X_d; \Bbbk)$, which is equivalent to
  $z \in B(X_d; \Bbbk)$.
\end{itemize}

Indeed, we can find chains
$\{\tilde{\sigma}_1, \ldots, \tilde{\sigma}_N \}$, and a decomposition of
indices $\{1, \ldots, N\} = D \sqcup D' \sqcup E$ satisfying the following
conditions, where $\hat{B}$ is the matrix returned by Algorithm~\ref{alg:pd}.
\begin{cond}\label{cond:phbasis} \ 
  \begin{enumerate}[(a)]
  \item $\{\tilde{\sigma}_1,\ldots, \tilde{\sigma}_k\}$ is a basis of $C(X_k)$
  \item $\partial \tilde{\sigma}_j \not = 0$ for $j \in D$
  \item $L_{\hat{B}}$ is a bijection from $D$ to $D'$ and
    $\partial \tilde{\sigma}_j = \tilde{\sigma}_{L_{\hat{B}}(j)}$ for any $j \in D$
  \item $\partial \tilde{\sigma}_i = 0$ for $i \in D' \sqcup E$
  \end{enumerate}
\end{cond}
From Condition \ref{cond:phbasis}, we have the following facts.
\begin{itemize}
\item $G_k = \{ [\tilde{\sigma}_i]_k \mid  1 \leq i \leq k, i \in D' \sqcup E,
  i \not = L_{\hat{B}}(j) \text{ for any } j \in D \cap \{1, \ldots, k\}\}$
  is a basis of $H(X_k)$
\item $H(X_{i - 1}) \oplus \left< [\tilde{\sigma}_{i}]_i\right> = H(X_{i})$
  for $i \in D' \sqcup E$
\item For any $i \in D', j \in D$ with $i = L_{\hat{B}}(j)$, the following
  holds.
  \begin{itemize}
  \item $\tilde{\sigma}_i \in Z(X_{i}) \backslash  Z(X_{i-1})$
  \item $[\tilde{\sigma}_{i}]_k \not = 0$ for any $i \leq k < j$
  \item $[\tilde{\sigma}_{i}]_k \in G_k$ for any $i \leq k < j$
  \item $[\tilde{\sigma}_{i}]_j = 0$
  \end{itemize}
  This corresponds to the interval indecomposable
  \begin{equation}
    I(i, j) = 0 \to \cdots \to 0 \to \overset{i}{\check{\Bbbk}}
    \xrightarrow{1} \cdots 
    \xrightarrow{1} \overset{j - 1}{\check{\Bbbk}}
    \to \overset{j}{\check{0}} \to \cdots \to 0,
  \end{equation}
  and a birth-death pair $(i, j)$
\item For $i \in E$, the following holds
  \begin{itemize}
  \item $\tilde{\sigma}_i \in Z(X_{i}) \backslash  Z(X_{i-1})$
  \item $[\tilde{\sigma}_{i}]_k \not = 0$ for any $i \leq k$
  \item $[\tilde{\sigma}_{i}]_k \in G_k$ for any $i \leq k$
  \end{itemize}
  This corresponds to the interval indecomposable
  \begin{equation}
    I(i, \infty) = 0 \to \cdots \to 0 \to \overset{i}{\check{\Bbbk}}
    \xrightarrow{1} \cdots \xrightarrow{1} \Bbbk,
  \end{equation}
  and a birth-death pair $(i, \infty)$
\end{itemize}
From these facts, it is straightforward that $D(\X)$ is given by
\begin{equation}
  D(\X) = \{ (L_{\hat{B}}(j), j) \mid j \in D \} \cup \{ (i, \infty) \mid i \in E\}.
\end{equation}

To prove the theorem, we show the following three facts.
\begin{fact}\label{fact:algpd-finite-steps}
  Algorithm~\ref{alg:pd} always terminates in finite steps.
\end{fact}
\begin{fact}\label{fact:algpd-basis}
  Algorithm~\ref{alg:pd} gives the chains
  $\{\tilde{\sigma}_1, \ldots, \tilde{\sigma}_N \}$ and a decomposition of
  indices $\{1, \ldots, N\} = D \sqcup D' \sqcup E$
  satisfying (a), (b), (c), and (d) in Condition~\ref{cond:phbasis}.
\end{fact}
\begin{fact}\label{fact:algpd-unique}
  The decomposition \eqref{eq:interval-decomposition} is unique
  if the decomposition exists.
\end{fact}

The proof of Fact~\ref{fact:algpd-finite-steps} is very similar to that of  Fact~\ref{fact:algfree-finite-steps}.
To show Fact~\ref{fact:algpd-basis} we apply
Lemma~\ref{lem:outerloop} for $n=N+1$ by replacing $\Zint$ with $\Bbbk$
and Algorithm~\ref{alg:allfree} with Algorithm~\ref{alg:pd}.
Fact~\ref{fact:algpd-unique} is a consequence of the Krull-Schmidt theorem. Here we show a more elementary proof by using 
persistent Betti numbers. From \eqref{eq:persistence-betti-number},
\eqref{eq:multiplicity-1}, and \eqref{eq:multiplicity-2}, we can show that
the multiplicity of each birth-death pair is unique if
the decomposition exists since the persistent Betti numbers only depend
on the rank of the maps $H(X_m; \Bbbk) \to H(X_n; \Bbbk)$ for all $m \leq n$
and the ranks themselves are independent of the interval decomposition.


\bibliographystyle{spmpsci}
\bibliography{references}

\begin{thebibliography}{10}
\providecommand{\url}[1]{{#1}}
\providecommand{\urlprefix}{URL }
\expandafter\ifx\csname urlstyle\endcsname\relax
  \providecommand{\doi}[1]{DOI~\discretionary{}{}{}#1}\else
  \providecommand{\doi}{DOI~\discretionary{}{}{}\begingroup
  \urlstyle{rm}\Url}\fi

\bibitem{phat}
Bauer, U., Kerber, M., Reininghaus, J., Wagner, H.: Phat – persistent
  homology algorithms toolbox.
\newblock Journal of Symbolic Computation \textbf{78}, 76 -- 90 (2017).
\newblock \doi{http://dx.doi.org/10.1016/j.jsc.2016.03.008}.
\newblock
  \urlprefix\url{http://www.sciencedirect.com/science/article/pii/S0747717116300098}.
\newblock Algorithms and Software for Computational Topology

\bibitem{BJDMC}
Boissonnat, J.D., Maria, C.: Computing persistent homology with various
  coefficient fields in a single pass.
\newblock In: A.S. Schulz, D.~Wagner (eds.) Algorithms - ESA 2014, pp.
  185--196. Springer Berlin Heidelberg, Berlin, Heidelberg (2014)

\bibitem{bubenik2019homological}
Bubenik, P., Milicevic, N.: Homological algebra for persistence modules.
\newblock arXiv preprint arXiv:1905.05744  (2019)

\bibitem{carlsson}
Carlsson, G.: Topology and data.
\newblock Bull. Amer. Math. Soc. \textbf{46}, 255--308 (2009).
\newblock \doi{10.1090/S0273-0979-09-01249-X}

\bibitem{Carlsson2008}
Carlsson, G., Ishkhanov, T., de~Silva, V., Zomorodian, A.: On the local
  behavior of spaces of natural images.
\newblock International Journal of Computer Vision \textbf{76}(1), 1--12 (2008)

\bibitem{virus}
Chan, J.M., Carlsson, G., Rabadan, R.: Topology of viral evolution.
\newblock Proceedings of the National Academy of Sciences \textbf{110}(46),
  18566--18571 (2013).
\newblock \doi{10.1073/pnas.1313480110}.
\newblock \urlprefix\url{http://www.pnas.org/content/110/46/18566.abstract}

\bibitem{chazal2009proximity}
Chazal, F., Cohen-Steiner, D., Glisse, M., Guibas, L.J., Oudot, S.Y.: Proximity
  of persistence modules and their diagrams.
\newblock In: Proceedings of the twenty-fifth annual symposium on Computational
  geometry, pp. 237--246 (2009)

\bibitem{Chen11persistenthomology}
Chen, C., Kerber, M.: Persistent homology computation with a twist.
\newblock In: Proceedings 27th European Workshop on Computational Geometry
  (2011)

\bibitem{cohen2007stability}
Cohen-Steiner, D., Edelsbrunner, H., Harer, J.: Stability of persistence
  diagrams.
\newblock Discrete \& Computational Geometry \textbf{37}(1), 103--120 (2007)

\bibitem{optimal-Day}
Dey, T.K., Hirani, A.N., Krishnamoorthy, B.: Optimal homologous cycles, total
  unimodularity, and linear programming.
\newblock SIAM J. Comput. \textbf{40}(4), 1026--1044 (2011).
\newblock \doi{10.1137/100800245}.
\newblock \urlprefix\url{http://dx.doi.org/10.1137/100800245}

\bibitem{alphashape2}
Edelsbrunner, H.: Smooth surfaces for multi-scale shape representation.
\newblock In: P.S. Thiagarajan (ed.) Foundations of Software Technology and
  Theoretical Computer Science, pp. 391--412. Springer Berlin Heidelberg,
  Berlin, Heidelberg (1995)

\bibitem{eh}
Edelsbrunner, H., Harer, J.: Computational topology: an introduction.
\newblock American Mathematical Soc. (2010)

\bibitem{elz}
Edelsbrunner, H., Letscher, D., Zomorodian, A.: Topological persistence and
  simplification.
\newblock Discrete {\&} Computational Geometry \textbf{28}(4), 511--533 (2002).
\newblock \doi{10.1007/s00454-002-2885-2}.
\newblock \urlprefix\url{https://doi.org/10.1007/s00454-002-2885-2}

\bibitem{alphashape1}
Edelsbrunner, H., M\"{u}cke, E.P.: Three-dimensional alpha shapes.
\newblock ACM Trans. Graph. \textbf{13}(1), 43--72 (1994).
\newblock \doi{10.1145/174462.156635}.
\newblock \urlprefix\url{http://doi.acm.org/10.1145/174462.156635}

\bibitem{gakhar2019k}
Gakhar, H., Perea, J.A.: K$\backslash$" unneth formulae in persistent homology.
\newblock arXiv preprint arXiv:1910.05656  (2019)

\bibitem{AT}
Hatcher, A.: Algebraic Topology.
\newblock Cambridge University Press (2002).
\newblock \urlprefix\url{https://pi.math.cornell.edu/~hatcher/AT/ATpage.html}

\bibitem{Hiraoka28062016}
Hiraoka, Y., Nakamura, T., Hirata, A., Escolar, E.G., Matsue, K., Nishiura, Y.:
  Hierarchical structures of amorphous solids characterized by persistent
  homology.
\newblock Proceedings of the National Academy of Sciences \textbf{113}(26),
  7035--7040 (2016).
\newblock \doi{10.1073/pnas.1520877113}.
\newblock \urlprefix\url{http://www.pnas.org/content/113/26/7035.abstract}

\bibitem{Hu2019}
Hu, X., Li, F., Samaras, D., Chen, C.: Topology-preserving deep image
  segmentation.
\newblock In: the Thirty-third Conference on Neural Information Processing
  Systems (NeurIPS) (2019)

\bibitem{PhysRevE.95.012504}
Ichinomiya, T., Obayashi, I., Hiraoka, Y.: Persistent homology analysis of
  craze formation.
\newblock Phys. Rev. E \textbf{95}, 012504 (2017).
\newblock \doi{10.1103/PhysRevE.95.012504}.
\newblock \urlprefix\url{https://link.aps.org/doi/10.1103/PhysRevE.95.012504}

\bibitem{doi:10.1080/10586458.2018.1473821}
Kahle, M., Lutz, F.H., Newman, A., Parsons, K.: Cohen–lenstra heuristics for
  torsion in homology of random complexes.
\newblock Experimental Mathematics \textbf{0}(0), 0--0 (2018).
\newblock \doi{10.1080/10586458.2018.1473821}.
\newblock \urlprefix\url{https://doi.org/10.1080/10586458.2018.1473821}

\bibitem{Kimura2018}
Kimura, M., Obayashi, I., Takeichi, Y., Murao, R., Hiraoka, Y.: Non-empirical
  identification of trigger sites in heterogeneous processes using persistent
  homology.
\newblock Scientific Reports \textbf{8}(1), 3553 (2018).
\newblock \doi{10.1038/s41598-018-21867-z}.
\newblock \urlprefix\url{https://doi.org/10.1038/s41598-018-21867-z}

\bibitem{Linial*2006}
Linial, N., Meshulam, R.: Homological connectivity of random 2-complexes.
\newblock Combinatorica \textbf{26}(4), 475--487 (2006).
\newblock \doi{10.1007/s00493-006-0027-9}.
\newblock \urlprefix\url{https://doi.org/10.1007/s00493-006-0027-9}

\bibitem{Milosavljevic:2011:ZPH:1998196.1998229}
Milosavljevi\'{c}, N., Morozov, D., Skraba, P.: Zigzag persistent homology in
  matrix multiplication time.
\newblock In: Proceedings of the Twenty-seventh Annual Symposium on
  Computational Geometry, SoCG '11, pp. 216--225. ACM, New York, NY, USA
  (2011).
\newblock \doi{10.1145/1998196.1998229}.
\newblock \urlprefix\url{http://doi.acm.org/10.1145/1998196.1998229}

\bibitem{Newman2019}
Newman, A.: Small simplicial complexes with prescribed torsion in homology.
\newblock Discrete \& Computational Geometry \textbf{62}(2) (2019)

\bibitem{Otter2017}
Otter, N., Porter, M.A., Tillmann, U., Grindrod, P., Harrington, H.A.: A
  roadmap for the computation of persistent homology.
\newblock EPJ Data Science \textbf{6}(1), 17 (2017).
\newblock \doi{10.1140/epjds/s13688-017-0109-5}.
\newblock \urlprefix\url{https://doi.org/10.1140/epjds/s13688-017-0109-5}

\bibitem{perea2015sliding}
Perea, J.A., Harer, J.: Sliding windows and persistence: An application of
  topological methods to signal analysis.
\newblock Foundations of Computational Mathematics \textbf{15}(3), 799--838
  (2015)

\bibitem{polterovich2017persistence}
Polterovich, L., Shelukhin, E., Stojisavljevi{\'c}, V.: Persistence modules
  with operators in morse and floer theory.
\newblock arXiv preprint arXiv:1703.01392  (2017)

\bibitem{granular}
Saadatfar, M., Takeuchi, H., Robins, V., Francois, N., Hiraoka, Y.: Pore
  configuration landscape of granular crystallization.
\newblock Nature Communications \textbf{8}, 15082 (2017).
\newblock \doi{10.1038/ncomms15082}.
\newblock \urlprefix\url{https://www.nature.com/articles/ncomms15082}

\bibitem{Storjohann}
Storjohann, A.: Algorithms for matrix canonical forms.
\newblock Ph.D. thesis, Swiss Federal Institute of Technology -- ETH (2000).
\newblock \urlprefix\url{https://cs.uwaterloo.ca/~astorjoh/publications.html}.
\newblock \url{https://cs.uwaterloo.ca/~astorjoh/publications.html}

\bibitem{zc}
Zomorodian, A., Carlsson, G.: Computing persistent homology.
\newblock Discrete {\&} Computational Geometry \textbf{33}(2), 249--274 (2005).
\newblock \doi{10.1007/s00454-004-1146-y}.
\newblock \urlprefix\url{https://doi.org/10.1007/s00454-004-1146-y}

\end{thebibliography}

\end{document}